\documentclass[a4paper,10pt]{article}

\usepackage{amsmath,amsfonts,amssymb,amsthm}
\usepackage{graphicx}
\usepackage{enumerate}
\usepackage{xcolor}
\usepackage{url}
\usepackage{tcolorbox}

\usepackage{hyperref}

\DeclareMathOperator{\im}{im} 

\DeclareMathOperator{\e}{e} 
\DeclareMathOperator{\id}{Id}
\DeclareMathOperator{\sgnvar}{sgnvar}

\newcommand{\R}{{\mathbb R}}
\newcommand{\N}{{\mathbb N}}
\newcommand{\Z}{{\mathbb Z}}

\newcommand{\dd}[2]{\frac{\text{d} #1}{\text{d} #2}}
\newcommand{\trans}{\mathsf{T}}
\newcommand{\ones}{1}

\DeclareMathOperator{\relint}{relint}

\newcommand{\dep}{d}
\newcommand{\rka}{l}
\newcommand{\mm}{m}

\newcommand{\mA}{A}
\newcommand{\mB}{B}
\newcommand{\cc}{c}
\newcommand{\IL}{I}
\newcommand{\JL}{J}
\newcommand{\LL}{L}
\newcommand{\setP}{P}

\newcommand{\mBp}{{\mathcal B}}
\newcommand{\mG}{G}
\newcommand{\mH}{H}

\newcommand{\yc}{\bar y}
\newcommand{\yp}{y'}

\newtheorem{thm}{Theorem}
\newtheorem{pro}[thm]{Proposition}
\newtheorem{lem}[thm]{Lemma}
\newtheorem{cor}[thm]{Corollary}
\theoremstyle{definition}

\newtheorem{exa}{Example}
\newtheorem{cla}[exa]{Class}

\newcommand\blfootnote[1]{%
  \begin{NoHyper}
  \renewcommand\thefootnote{}\footnote{#1}%
  \addtocounter{footnote}{-1}%
  \end{NoHyper}
}

\begin{document}

\title{
Parametrized systems of generalized polynomial equations: first applications to fewnomials}

\author{%
Stefan M\"uller$$,
Georg Regensburger
}

\date{\today}

\maketitle

\begin{abstract}
We consider positive solutions to parametrized systems of generalized polynomial equations (with real exponents and positive parameters).
By a fundamental result obtained in parallel work,
polynomial systems are determined by geometric objects, rather than matrices:
a polytope $P$ (arising from the coefficient matrix)
and two subspaces representing monomial differences and dependencies (arising from the exponent matrix).
The dimension of the latter subspace, the monomial dependency $d$, is crucial. 
Indeed, we rewrite {\em polynomial} equations 
in terms of $d$ {\em binomial} equations on the coefficient polytope $P$,
involving $d$ monomials in the parameters.
We further study the solution set on $P$ using methods from analysis
such as sign-characteristic functions 
and Wronskians.

In this work, we present first applications to fewnomial systems
through five (classes of) examples.
In particular, we study
(i) $n$ trinomials involving ${n+2}$ monomials in $n$ variables, having dependency $d=1$,
and (ii) one trinomial and one $t$-nomial (with $t\ge3$) in two variables, having $d=t-1\ge2$.
For (i), we bound the number of positive solutions
using the number of roots of a univariate polynomial of degree at most $n$. 
We also show that this number is always less than or equal to the number of sign changes in an optimal Descartes' rule
given in Bihan et al.~(2021).
For (ii), we improve upper bounds given in Li et al.~(2003) and Koiran et al.~(2015).
Further, for two trinomials ($t=3$),
we refine the known upper bound of five in terms of the exponents,
and we find an example with five positive solutions
that is even simpler than the smallest ``Haas system''.

\vspace{2ex}
\end{abstract}

\vspace{-2ex}
\blfootnote{
\scriptsize

\noindent
{\bf Stefan~M\"uller} \\
Faculty of Mathematics, University of Vienna, Oskar-Morgenstern-Platz 1, 1090 Wien, Austria \\[1ex]
{\bf Georg Regensburger} \\
Institut f\"ur Mathematik, Universit\"at Kassel, Heinrich-Plett-Strasse 40, 34132 Kassel, Germany \\[1ex]
Corresponding author: 
\href{mailto:st.@univie.ac.at}{st.mueller@univie.ac.at}
}


\section{Introduction}

In this work,
we apply our fundamental results on positive solutions to pa\-ram\-e\-trized systems of generalized polynomial {\em inequalities} 
(with real exponents and positive parameters), obtained in parallel work~\cite{MuellerRegensburger2023a},
to generalized polynomial {\em equations}, specifically to fewnomial systems.

Let 
$\mA \in \R^{\rka \times \mm}$ be a coefficient matrix,
$\mB \in \R^{n \times \mm}$ be an exponent matrix, and
$\cc \in \R^\mm_>$ be a positive parameter vector.
They define the
parametrized system of generalized polynomial equations
\begin{equation*}
\sum_{j=1}^{\mm} a_{ij} \, \cc_j \, x_1^{b_{1j}} \cdots x_n^{b_{nj}} = 0 , \quad i=1,\ldots,\rka ,
\end{equation*}
in $n$ positive variables $x_i>0$, $i=1,\ldots,n$,
and involving $m$ monomials $x_1^{b_{1j}} \cdots x_n^{b_{nj}}$, $j=1,\ldots,\mm$.
In compact form, 
\begin{equation} \label{eq:problem}
\mA \left( \cc \circ x^\mB \right) = 0
\end{equation}
for $x \in \R^n_>$.

We obtain~\eqref{eq:problem} as follows.
From the exponent matrix $\mB = (b^1,\ldots,b^\mm)$,
we define the monomials $x^{b^j} = x_1^{b_{1j}} \cdots x_n^{b_{nj}} \in \R_>$,
the vector of monomials $x^\mB \in \R^\mm_>$ via $(x^\mB)_j = x^{b^j}$,
and the vector of monomial terms $\cc \circ x^\mB \in \R^\mm_>$ using the componentwise product~$\circ$.
(All notation is formally introduced at the end of this introduction.)

System~\eqref{eq:problem} allows for applications in two areas:
(i) fewnomial systems, see e.g.\ \cite{Khovanskii1991,Sottile2011},
and (ii) reaction networks with (generalized) mass-action kinetics,
see e.g.\ \cite{HornJackson1972,Horn1972,Feinberg1972} and \cite{MuellerRegensburger2012,MuellerRegensburger2014,Mueller2016,MuellerHofbauerRegensburger2019}.
In (ii), every reaction rate is given by a positive parameter (a rate constant) times a monomial.
We depict the situation in Figure~\ref{fig:1}.

\begin{figure}[h]
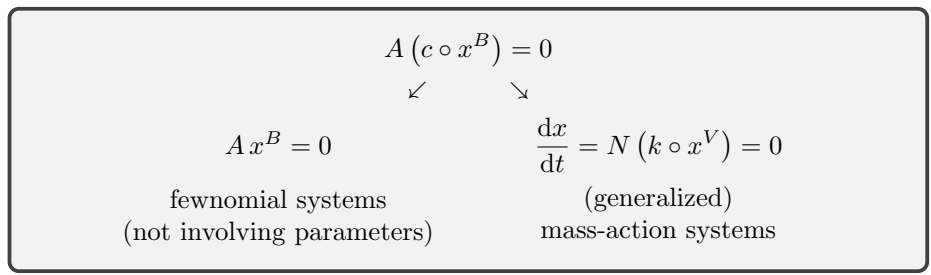

\begin{tcolorbox}
\vspace{-3ex}
\begin{gather*}
\mA \left(\cc \circ x^\mB\right) = 0 \\ 
\rotatebox[origin=c]{-45}{$\downarrow$} \hspace{1cm} \rotatebox[origin=c]{45}{$\downarrow$} \\
\parbox[t]{5cm}{\centering $\mA \, x^\mB = 0$ \\[2ex] fewnomial systems \\ (not involving parameters)}
\parbox[t]{5cm}{\centering $\displaystyle \dd{x}{t} = N \left(k \circ x^V\right) = 0$ \\[1ex] (generalized) \\ mass-action systems}
\end{gather*}
\end{tcolorbox}
\label{fig:1}
\caption{(Parametrized) systems of generalized polynomial equations for positive variables} 
\end{figure}

Motivated by reaction networks,
we extend the standard setting of fewnomial systems in two ways.
(i) We assign a positive {\em parameter}~$\cc_j$ to every monomial~$x^{b^j}$, and
(ii) we consider partitions of the monomials into {\em classes},
corresponding to a decomposition of $\ker \mA$ as a direct product.

In order to study system~\eqref{eq:problem},
we instantiate our results for parametrized systems of generalized polynomial {\em inequalities},
presented in parallel work~\cite{MuellerRegensburger2023a}.
\begin{enumerate}
\item
We identify the invariant geometric objects of system~\eqref{eq:problem},
namely, the {\em coefficient ``polytope''}~$\setP$, 
the monomial difference subspace~$\LL$,
and the monomial dependency subspace~$D$.

More specifically,
since $\cc \circ x^\mB \in \R^\mm_>$,
we write system~\eqref{eq:problem} as $(\cc \circ x^\mB) \in C$,
where $C = \ker \mA \cap \R^\mm_>$ is the coefficient cone.
We obtain the coefficient polytope $\setP = C \cap \Delta$ 
by intersecting $C$ with a direct product $\Delta$ of simplices (or appropriate affine subspaces) on the classes.

The monomial difference subspace $L$ is determined by the affine monomial spans of the classes,
whereas the monomial dependency subspace $D$ captures affine dependencies within {\em and} between classes.
In particular, the {\em monomial dependency}~$\dep = \dim D$ is crucial.
\item
Most importantly,
we rewrite system~\eqref{eq:problem} in terms of $\dep$ binomial equations on the coefficient polytope $\setP$,
involving $\dep$ monomials in the parameters.

In particular, we establish an {\em explicit} bijection between
the solution set \[ Z_\cc = \{ x \in \R^n_> \mid \mA \left( \cc \circ x^\mB \right) = 0 \} \] of system~\eqref{eq:problem} 
and the solution set on $P$,
 \[ Y_\cc = \{ y \in P \mid y^z = \cc^z \text{ for all } z \in D \} , \]
via exponentiation.
\item
We obtain a problem classification.
If $\dep=0$ (the ``very few''-nomial case), solutions exist (for all positive  $\cc$) and can be parametrized explicitly.

If $\dep>0$, we further study the solution set on $\setP$ 
using methods from analysis
such as sign-characteristic functions, Descartes' rule of signs for functions, and Wronskians.
\end{enumerate}

Our results in~\cite{MuellerRegensburger2023a} lay the groundwork for a novel approach to ``positive algebraic geometry''.
They are based on methods from linear algebra and convex/polyhedral geometry
(and complemented by techniques from analysis).

A fundamental problem in fewnomial theory
is to bound the number of the (finitely many) solutions
to $n$ equations for $n$ unknowns.
In 1980, Khovanskii gave an upper bound for the number of non-degenerate positive
solutions to polynomial equations in $n$ variables involving $\mm$ monomials in terms of~$n$ and~$\mm$~\cite{Khovanskii1991,Sottile2011}.
The general bound has been further improved by Bihan and Sottile,
by rewriting fewnomial systems as so-called Gale dual systems~\cite{BihanSottile2007}.
For a more geometric formulation of the latter, see also~\cite{BihanDickensteinMagali2020}.

Gale dual systems, as used e.g.\ in \cite{BihanSottile2007,Bihan2017,BihanDickensteinMagali2020,Bihan2021}, 
do not involve the invariant geometric objects (the coefficient cone and polytope) explicitly.
In particular, one chooses a basis of $\ker \mA$
to parametrize the coefficient cone $C = \ker \mA \cap \R^\mm_>$.
In our approach, we abstractly rewrite the problem in terms of the coefficient polytope~$\setP$,
and for explicit computations, we often parametrize it via generators. 



In this work, we present first applications of our approach to real fewnomial theory
through five (classes of) examples,
all of which involve trinomials.
To begin with, we study an underdetermined system of two trinomials in three variables (having two classes)
and parametrize the components of the (infinitely many) solutions,
using sign-characteristic functions (defined in this work) and their roots.
Most importantly, we study
(i) $n$ trinomials involving ${n+2}$ monomials in $n$ variables 
and (ii) one trinomial and one $t$-nomial in two variables. 
For (i), 
we bound the number of positive solutions
by one plus the number of roots of a univariate polynomial of degree at most $n$ on the interval $(-1,1)$. 
We also show that this number is always less than or equal to the number of sign changes in an optimal Descartes' rule
given in Bihan et al.~\cite{Bihan2021}.
For (ii), 
we improve upper bounds given in Li et al.~\cite{Li2003} and Koiran et al.~\cite{Koiran2015a}. Further, for two trinomials ($t=3$),
we refine the known upper bound of five in terms of the exponents,
and we find an example with five positive solutions
that is even simpler than the smallest ``Haas system''~\cite{Haas2002}. 

We foresee applications of our approach to many more problems
such as 
existence and uniqueness of positive solutions, for given or all parameters, 
upper bounds on the number of solutions/components beyond trinomials,
and extensions of classical results from reaction network theory. 
For the most general setting of generalized polynomial {\em inequalities},
see the parallel work~\cite{MuellerRegensburger2023a}.
For a characterization of the existence of a unique solution for all parameters
(and a resulting multivariate Descartes' rule of signs), see the very recent work~\cite{DeshpandeMueller2024}.

{\bf Organization of the work.} 
In Section~\ref{sec:help}, 
we formally introduce the geometric objects and auxiliary matrices
required to rewrite system~\eqref{eq:problem}.
In Section~\ref{sec:main}, we instantiate our main result for inequalities in~\cite{MuellerRegensburger2023a} to equations,
in particular, we state Theorem~\ref{thm:main}, 
and in~\ref{sec:dim}, we consider the dimensions of all geometric objects.
Finally, in Section~\ref{sec:app},
we apply our results to three examples and two classes of examples,
all of which involve trinomials.
We briefly summarize all examples at the beginning of the section.
Our main results are 
Theorem~\ref{thm:exa_general} for $n$ trinomials involving $n+2$ monomials in $n$ variables,
Theorem~\ref{thm:tnomial} for one trinomial and one $t$-nomial in two variables,
and Theorem 
\ref{thm:tritri_exp} for two trinomials.

In Appendix~\ref{app:sc}, we introduce and analyze sign-characteristic functions,
and in Appendix~\ref{app:Rolle}, we elaborate on Rolle's Theorem.


\subsection*{Notation}

We denote the positive real numbers by $\R_>$ and the nonnegative real numbers by $\R_\ge$.
We write $x>0$ for $x \in \R^n_>$ and $x \ge 0$ for $x \in \R^n_\ge$.
For $x \in \R^n$, we define $\e^x \in \R^n_>$ componentwise, and for $x \in \R^n_>$, we define $\ln x \in \R^n$.

For $x \in \R^n_>, \, y \in \R^n$, we define the generalized monomial $x^y = \prod_{i=1}^n (x_i)^{y_i} \in \R_>$,
and for $x \in \R^n_>, \, Y = ( y^1 \ldots y^\mm) \in \R^{n \times \mm}$, 
we define the vector of monomials $x^Y \in \R^\mm_>$ via $(x^Y)_i =x^{y^i}$. 

For $x,y \in \R^n$, we denote their scalar product by $x \cdot y \in \R$
and their componentwise (Hadamard) product by $x \circ y \in \R^n$,
and for 
\[
\alpha \in \R^\ell , \quad x = \begin{pmatrix} x^1 \\ \vdots \\ x^\ell \end{pmatrix} \in \R^n
\]
with $x^1 \in \R^{n_1}$, \ldots, $x^\ell \in \R^{n_\ell}$ and $n = n_1 + \cdots + n_\ell$,
we introduce
\[
\alpha \odot x = \begin{pmatrix} \alpha_1 \, x^1 \\ \vdots \\ \alpha_\ell \, x^\ell \end{pmatrix} \in \R^n .
\]

We write $\id_n \in \R^{n \times n}$ for the identity matrix and $\ones_n \in \R^n$ for the vector with all entries equal to one.
Further, for $n >1$, we introduce the incidence matrix $\IL_n \in \R^{n \times (n-1)}$ of the (star-shaped) graph 
$n \to 1, \, n \to 2, \, \ldots, \, n \to (n-1)$
with $n$ vertices and $n-1$ edges, 
\[
\IL_n = \begin{pmatrix} \id_{n-1} \\ - \ones_{n-1}^\trans \end{pmatrix} , \quad \text{that is,} \quad
\IL_2 = \begin{pmatrix} 1 \\ - 1 \end{pmatrix} , \;
\IL_3 = \begin{pmatrix} 1 & 0 \\ 0 & 1 \\ -1 & -1 \end{pmatrix} , \text{ etc.}
\]
Clearly, $\ker \IL_n = \{0\}$ and $\ker \IL_n^\trans = \im \ones_n$. 


\section{Geometric objects and auxiliary matrices} \label{sec:help}

From the introduction,
recall the coefficient matrix $\mA \in \R^{\rka \times \mm}$,
the exponent matrix $\mB \in \R^{n \times \mm}$,
and the positive parameter vector $\cc \in \R^\mm_>$.
We introduce geometric objects and auxiliary matrices
required to reformulate system~\eqref{eq:problem}. 

Since $\cc \circ x^\mB \in \R^\mm_>$ for $x \in \R^n_>$, 
we write $\mA \, ( \cc \circ x^\mB ) = 0$ as
\begin{equation}
\left( \cc \circ x^\mB \right) \in C
\end{equation}
with the convex cone
\begin{subequations}
\begin{equation}
C = \ker \mA \cap \R^\mm_> .
\end{equation}
We call $C$ the {\em coefficient cone}.
The closure of $C$ is a polyhedral cone,
\begin{equation*}
\overline C = \ker \mA \cap \R^\mm_\ge ,
\end{equation*}
and
$C = \overline C \cap \R^\mm_> = \relint \overline C$.
In~\cite{MuellerRegensburger2016}, $\overline C$ is called an s-cone (or subspace cone).
Throughout this work, we assume that $C$ is non-empty (as a necessary condition for the existence of solutions)
and that $\mA$ has full row rank.
Further, we assume that
\begin{equation}
\mA = \begin{pmatrix} \mA_1 & \ldots & \mA_\ell \end{pmatrix} \in \R^{\rka \times \mm}
\end{equation}
has $\ell \ge 1$ blocks $\mA_i \in \R^{\rka \times \mm_i}$ with $\mm_1+\ldots+\mm_\ell = \mm$ 
such that $C$
is the direct product of the convex cones
\begin{equation}
C_i = \ker \mA_i \cap \R^{\mm_i}_> ,
\end{equation}
that is, \begin{equation}
C = 
C_1 \times \cdots \times C_\ell .
\end{equation}

\end{subequations}

Accordingly, 
\begin{equation}
\mB = \begin{pmatrix} \mB_1 & \ldots & \mB_\ell \end{pmatrix} \in \R^{n \times \mm}
\end{equation}
with $\ell$ blocks $\mB_i \in \R^{n \times \mm_i}$ and 
\begin{equation}
\cc= \begin{pmatrix} \cc^1 \\ \vdots \\ \cc^\ell \end{pmatrix} \in \R^\mm_>
\end{equation} 
with $\cc^i \in \R^{\mm_i}_>$.

The blocks of $\mA = (\mA_1 \,\ldots\, \mA_\ell) \in \R^{\rka \times \mm}$
induce a partition of the indices $\{1,\ldots,\mm\}$ into $\ell$ {\em classes}.
Correspondingly, the columns of $\mB=(b^1,\ldots,b^\mm)$, 
and hence the monomials $x^{b^j}$, $j=1,\ldots,\mm$
are partitioned into classes.
Our main results hold for any partition that permits a direct product form of $C$,
but the finest partition allows us to maximally reduce the problem dimension.

Indeed,
going from a cone to a bounded set reduces dimensions by one per class.
Hence, we introduce the direct product
\begin{subequations}
\begin{equation}
\Delta = \Delta_{\mm_1-1} \times \cdots \times \Delta_{\mm_\ell-1} 
\end{equation}
of the standard simplices
\begin{equation}
\Delta_{\mm_i-1} = \{ y \in \R^{\mm_i}_\ge \mid \ones_{\mm_i} \cdot y = 1 \} 
\end{equation}
\end{subequations}
and define the bounded set
\begin{subequations}
\begin{equation}
\setP = C \cap \Delta .
\end{equation}
For computational simplicity in examples, 
we often use sets of the form $\tilde \Delta_{\mm_i-1} = \{ y \in \R^{\mm_i}_\ge \mid u \cdot y = 1 \}$ 
for some vector $u \in \relint C_i^*$, where $C_i^*$ is the dual cone of $C_i$.

Clearly,
\begin{equation}
\setP = P_1 \times \cdots \times P_\ell
\end{equation}
with
\begin{equation}
P_i = C_i \cap \Delta_{\mm_i-1} .
\end{equation}
We call $\setP$ the {\em coefficient polytope}. In fact, $\setP$ is a polytope without boundary. Strictly speaking, only its closure $\overline \setP$ is a polytope.
\end{subequations}

Now, let
\begin{equation}
\IL = \begin{pmatrix} \IL_{\mm_1} & & 0 \\ & \ddots & \\ 0 & & \IL_{\mm_\ell} \end{pmatrix} \in \R^{\mm \times (\mm-\ell)}
\end{equation}
be the $\ell \times \ell$ block-diagonal (incidence) matrix
with blocks $\IL_{\mm_i} \in \R^{\mm_i \times (\mm_i-1)}$
and let
\begin{subequations}
\begin{equation}
M = \mB \, \IL \in \R^{n\times(\mm-\ell)} .
\end{equation}
Clearly, $\ker \IL = \{0\}$ and $\ker \IL^\trans = \im \ones_{\mm_1} \times \cdots \times \im \ones_{\mm_\ell}$.
The matrix
\begin{equation}
M = \begin{pmatrix} \mB_1 \IL_{\mm_1} & \ldots &  \mB_\ell \IL_{\mm_\ell} \end{pmatrix}
\end{equation}
\end{subequations}
records the differences of the columns of $B$ within classes,
explicitly, between the first $\mm_i-1$ columns of $B_i$ and its last column, for all $i=1,\ldots,\ell$. 
Hence,
\begin{equation}
\LL = \im M \subseteq \R^n
\end{equation}
is the sum of the linear subspaces associated with the affine spans of the columns of $B$ in the $\ell$ classes.
We call $\LL$ the {\em monomial difference subspace}.
Further, we call
\begin{subequations}
\begin{equation}
\dep = \dim (\ker M)
\end{equation}
the {\em monomial dependency}.
It can be determined as
\begin{equation} \label{eq:dep}
\dep = m-\ell-\dim \LL ,
\end{equation}
\end{subequations}
cf.~\cite[Proposition~1]{MuellerRegensburger2023a}.

We call a system {\em generic} if $M \in \R^{n \times (\mm-\ell)}$ has full (row or column) rank.
A system is generic if and only if $\LL=\R^n$ or $\dep=0$,
cf.~\cite[Proposition~2]{MuellerRegensburger2023a}.

In real fewnomial theory,
it is standard to assume $\LL = \R^n$. Otherwise, dependent variables can be eliminated.
A system with $\dep=0$ (a ``very few''-nomial system) allows for an explicit parametrization of the solution set,
see~Corollary~\ref{cor:d0} below.

To conclude our exposition,
let
\begin{equation}
\JL = \begin{pmatrix} \ones_{\mm_1}^\trans & & 0 \\ & \ddots & \\ 0 & & \ones_{\mm_\ell}^\trans \end{pmatrix} \in \R^{\ell \times \mm}
\end{equation}
be the $\ell \times \ell$ block-diagonal ``Cayley'' matrix
with blocks $\ones_{\mm_i}^\trans \in \R^{1 \times \mm_i}$
and
\begin{equation}
\mBp = \begin{pmatrix} \mB \\ \JL \end{pmatrix} \in \R^{(n+\ell) \times \mm} .
\end{equation}
Clearly, $\JL \, \IL = 0$, in fact, $\ker \JL = \im \IL$.
We call 
\begin{subequations}
\begin{equation}
D = \ker \mBp 
\subset \R^\mm
\end{equation}
the {\em monomial dependency subspace}.

Finally,
let $\mG \in \R^{(\mm-\ell) \times \dep}$ represent a basis for $\ker M$, that is, $\im \mG = \ker M$ (and $\ker \mG = \{0\}$),
and let $\mH = \IL \, \mG \in \R^{\mm \times \dep}$. 
Then $\mH$ represents a basis of $\ker \mBp$,
that is, 
\begin{equation}
D = \ker \mBp = \im \mH 
\quad (\text{and }
\dim D = \dep) ,
\end{equation}
\end{subequations}
cf.~\cite[Lemma~4]{MuellerRegensburger2023a}.

For illustrations of all geometric objects and auxiliary matrices,
we refer the reader to 
Example~\ref{exa:tri3d} (two trinomials in three variables with $\dep=1$ and $\ell=2$ classes),
Example~\ref{exa:Bihan} (two overlapping trinomials in two variables with $\dep=1$ and $\ell=1$), and
Example~\ref{exa:tritri} (two trinomials in two variables with $\dep=2$ and $\ell=2$),
in Section~\ref{sec:app}.


\section{Main result for equations} \label{sec:main}

Using the geometric objects and auxiliary matrices introduced in Section~\ref{sec:help},
we instantiate \cite[Theorem~5]{MuellerRegensburger2023a} for parametrized systems of generalized polynomial inequalities
(from inequalities to equations).

\begin{thm} \label{thm:main}
Consider the parametrized system of generalized polynomial equations $\mA \, (\cc \circ x^\mB) = 0$ for the positive variables $x \in \R^n_>$,
given by a real coefficient matrix $\mA \in \R^{\rka \times \mm}$,
a real exponent matrix $\mB \in \R^{n \times \mm}$, 
and a positive parameter vector $\cc \in \R^\mm_>$.
The solution set $Z_\cc = \{ x \in \R^n_> \mid \mA \, (\cc \circ x^\mB) = 0 \}$ 
can be written as
\[
Z_\cc = \{ (y \circ \cc^{-1})^E \mid y \in Y_\cc \} \circ \e^{\LL^\perp}
\quad \text{with} \quad
Y_\cc = \{ y \in \setP \mid y^z = c^z \text{ for all } z \in D \} .
\]
Here, $\setP$ is the coefficient polytope, $D$ is the monomial dependency subspace, $\LL$ is the monomial difference subspace,
and the matrix $E = \IL \, M^*$ is given by the (incidence) matrix $\IL$ and a generalized inverse $M^*$ of $M = B \, \IL$.
\end{thm}

Theorem~\ref{thm:main} can be read as follows:
In order to determine the solution set $Z_\cc = \{ x \in \R^n_> \mid \mA \, (\cc \circ x^\mB) = 0 \}$,
first determine the {\em solution set on the coefficient polytope},
\begin{equation} \label{eq:solpol}
Y_\cc = \{ y \in P \mid y^z = \cc^z \text{ for all } z \in D \} .
\end{equation}
The coefficient polytope $\setP$ is determined by the coefficient matrix~$\mA$ (and its classes),
and the dependency subspace $D$ is determined by the exponent matrix~$\mB$ ({\em and} the classes of $\mA$).
Explicitly, using $D = \im \mH$ with $\dep = \dim D$, there are $\dep$ binomial equations \begin{equation} y^{\mH} = \cc^{\mH} , \end{equation}
which depend on the parameters via $\dep$ monomials $\cc^{\mH}$.

To a solution $y \in Y_\cc$ on the coefficient polytope,
there corresponds the actual solution $x = (y \circ c^{-1})^E \in Z_\cc$.
In fact, if (and only if) $\dim L < n$, 
then $y \in Y_\cc$ corresponds to an exponential manifold of solutions, $x \circ \e^{L^\perp} \subset Z_\cc$.
By~\cite[Proposition~6]{MuellerRegensburger2023a},
the set $Z_\cc / e^{\LL^\perp}$, that is, the equivalence classes of $Z_\cc$ (given by $x' \sim x$ if $x'=x \circ \e^v$ for some $v \in L^\perp$),
and the set $Y_\cc$ are in one-to-one correspondence.
\begin{pro}
There is a bijection between $Z_\cc / e^{\LL^\perp}$ and $Y_\cc$.
\end{pro}

An important special case arises if $\dep=0$ (the ``very few''-nomial case) and hence $Y_\cc = P$.
Then the solution set $Z_\cc$ in Theorem~\ref{thm:main}
has an explicit parametrization.

\begin{cor} \label{cor:d0}
If $\dep=0$,
then
\[
Z_\cc = \{ (y \circ \cc^{-1})^E \mid y \in P \} \circ \e^{\LL^\perp} ,
\]
in particular, the parametrization is explicit,
and there exists a solution for all~$\cc$.

If $\mA \, (\cc \circ x^\mB)=0$ is a system of $n$~equations in $n$~unknowns,
then there exists a unique solution for all $\cc$.
\end{cor}
\begin{proof}
The first statement is an instance of Theorem~\ref{thm:main}.
It remains to prove the second statement.
By Equation~\eqref{eq:dep}, $\dep = \mm-\ell -\dim \LL = 0$.
Hence $\dim \LL = \mm-\ell$ and $\dim \LL^\perp = n - (\mm-\ell) \ge 0$.
By the dimensional considerations below,  $\dim P = \mm - \ell - \rka \ge 0$.
Hence, $\dim P + \dim \LL^\perp = n-\rka$.
If $\rka=n$, that is, if $\mA \, (\cc \circ x^\mB)=0$ is a system of $n$~equations in $n$~unknowns,
then $\dim P = \dim \LL^\perp = 0$,
and
there exists a unique solution for all $\cc$.
\end{proof}


\subsection{Dimensional considerations} \label{sec:dim}

The crucial quantities that appear in (the dimensions of) the geometric objects are the number of variables $n$,
the number of equations $\rka$,
the number of monomials $\mm$, and the number of classes $\ell$.

First, we consider the dimensions of the coefficient cone and polytope, $C$ and~$\setP$, 
arising from the coefficient matrix $\mA \in \R^{\rka \times \mm}$.
Recall that $\mA$ has full row rank and $C = \ker A \cap \R^\mm_> \neq \emptyset$.
Hence, 
\begin{subequations}
\begin{equation}
\dim C = \dim (\ker \mA) = \mm - \rka .
\end{equation}
Going from the cone to the polytope reduces dimensions by one per class.
Hence,
\begin{equation}
\dim P = \dim C - \ell= \mm - \rka - \ell .
\end{equation}
\end{subequations}

Second, we consider the dimensions of the monomial difference/dependency subspaces, $\LL$ and $D$, 
arising from the exponent matrix $\mB \in \R^{n \times \mm}$.
The following case distinction illuminates the interplay of $\dim \LL$ and the monomial dependency $\dep=\dim D$.
Here, we assume the generic case where $M \in \R^{n \times (\mm-\ell)}$ has full (row or column) rank,
that is, $\dim \LL = \min(n,\mm-\ell)$.
\begin{itemize}
\item
$\mm \le n+\ell$:

Here, $\dim \LL = \mm -\ell$ and hence $\dep = \mm-\ell-\dim \LL = 0$.
This is the monomial dependency zero (or ``very few''-nomial) case.

We have a further case distinction.
\begin{itemize}
\item[]
$\mm < n+\ell$: $\dim \LL < n$ (that is, $\LL^\perp\neq\{0\}$).
\item[]
$\mm = n+\ell$: $\dim \LL = n$ (that is, $\LL^\perp=\{0\}$).
\end{itemize}
\item
$\mm > n+ \ell$:

Here, $\dim \LL = n$ (that is, $\LL^\perp=\{0\}$) and hence $\dep = \mm-\ell-n > 0$.
\end{itemize}

As a consequence, for generic systems with non-zero monomial dependency,
there is no exponential/monomial parametrization involved.
\begin{cor}
If a system is generic and $\dep>0$, then
\[
Z_\cc = \{ (y \circ \cc^{-1})^E \mid y \in Y_\cc \} 
\quad \text{with} \quad
Y_\cc = \{ y \in P \mid y^z = \cc^z \text{ for all } z \in D \} .
\]
\end{cor}

If a system is generic and $\dep>0$, then $\dim \LL=n$ and hence $\dep =\mm-\ell-n$.
Further, $\dim P = \mm - \ell - \rka \ge 0$.
If $\rka=n$,
that is, if $\mA \, (\cc \circ x^\mB)=0$ is a system of $n$~equations in $n$~unknowns,
then $\dim P = \dep$,
and there are $\dep$ binomial equations on a $\dep$-dimensional coefficient polytope.
Still, $Y_\cc$
(the solution set on the coefficient polytope) 
can be empty, finite, or infinite
(unlike for $\dep=0$).


\section{Applications to fewnomials} \label{sec:app}

We apply our fundamental results on parametrized systems of generalized polynomial equations to fewnomial systems.
In particular, 
we study three examples and two classes of examples,
all of which involve trinomials.
We choose them to cover several ``dimensions'' of the problem.
In particular, we consider examples 
\begin{itemize}
\item 
with monomial dependency $\dep=1$ or $\dep \ge 2$,
\item
with $\ell=1$ or $\ell=2$ classes,
and
\item
with finitely or infinitely many solutions, thereby providing bounds on the number of solutions or components.
\end{itemize}

Ordered by the dependency,
we study the following fewnomial systems:
\begin{itemize}
\item $\dep=1$:
\begin{itemize}
\item
Example~\ref{exa:tri3d}: \\
{\bf Two trinomials in three variables}

We parametrize the components of the (infinitely many) positive solutions,
using sign-characteristic functions (defined in this work).
\item
Example~\ref{exa:Bihan}: \\
{\bf Two overlapping trinomials involving four monomials in two variables}

We bound the number of positive solutions by one plus the number of roots of a univariate polynomial of degree at most two on the interval $(-1,1)$.
We also show that this number (which depends on both exponents and coefficients)
is always less than or equal to the number of sign changes in an optimal Descartes' rule (which depends on exponents and orders of coefficients).
In concrete systems, we demonstrate that our new upper bound improves the Descartes' rule.

\item
Class~\ref{cla:Bihan}: \\
{\bf $n$ overlapping trinomials involving $(n+2)$ monomials in $n$ variables}

We extend the result for the case $n=2$ in Example~\ref{exa:Bihan} to the general class.

\end{itemize}
\item $\dep\ge2$:
\begin{itemize}
\item
Class~\ref{cla:tri-t}: \\
{\bf One trinomial and one $t$-nomial in two variables} 

We improve known upper bounds 
in terms of $t$.

\item
Example~\ref{exa:tritri}: \\
{\bf Two trinomials in two variables}

We reconsider the simplest case $t=3$ in Class~\ref{cla:tri-t}.

We refine the known upper bound of five (for the number of finitely many positive solutions) 
in terms of the exponents,
and we find an example with five positive solutions
that is even simpler than the smallest ``Haas system''. 
\end{itemize}
\end{itemize}


\subsection{Dependency one}

\begin{exa} \label{exa:tri3d}
We consider {\bf two trinomials in three variables}, that is,
\begin{align*}
c_1 \, x^{b^1} + c_2 \, x^{b^2} - 1 &= 0 , \\ 
c_3 \, x^{b^3} + c_4 \, x^{b^4} - 1 &= 0 ,
\end{align*}
with $x \in \R^3_>$ and $b^1, b^2, b^3, b^4 \in \R^3$.
In compact form,
$\mA \, (\cc \circ x^\mB)$ with
\[
\mA = \begin{pmatrix} 1 & 1 & -1 & 0 & 0 & 0 \\ 0 & 0 & 0 & 1 & 1 & -1 \end{pmatrix} , \quad
\mB = \begin{pmatrix} 
b^1& b^2 & 0 & b^3 & b^4 & 0
\end{pmatrix} ,
\quad
\cc = \begin{pmatrix} c_1 \\ c_2 \\ 1 \\ c_3 \\ c_4 \\ 1 \end{pmatrix} .
\]
Clearly, $\mm=6$, $n=3$, and $\ell=2$. 

Crucial coefficient data, arising from $\mA = \begin{pmatrix} \mA_1 & \mA_2 \end{pmatrix}$ with $\mA_i \in \R^{2 \times 3}$, are
\[
C = C_1 \times C_2
\quad \text{with} \quad
C_i = \ker \mA_i \cap \R^3_>
\quad \text{and} \quad
\ker \mA_i = \im \begin{pmatrix} 1 & 0 \\ 0 & 1 \\ 1 & 1 \end{pmatrix} ,
\]
and further
\[
\setP = P_1 \times P_2
\quad \text{with} \quad
P_i = C_i \cap \Delta_2 ,
\]
where we use the (non-standard) simplex $\Delta_2 = \{ y \in \R^3_\ge \mid \ones_3 \cdot y = 2 \}$.
Hence, for $y \in P$, we have
\[
y = \begin{pmatrix} y^1 \\ y^2 \end{pmatrix}
\quad \text{with} \quad
y^i = \lambda_i \begin{pmatrix} 1 \\ 0 \\ 1 \end{pmatrix} + (1-\lambda_i) \begin{pmatrix} 0 \\ 1 \\ 1 \end{pmatrix}
= \begin{pmatrix} \lambda_i \\ 1-\lambda_i \\ 1 \end{pmatrix}
\]
and $\lambda_1, \lambda_2 \in (0,1)$.
We assume the generic case, where
\[
M = \mB \, \begin{pmatrix} \IL_3 & 0 \\ 0 & \IL_3 \end{pmatrix} = \begin{pmatrix} 
b^1& b^2 & b^3 & b^4
\end{pmatrix}
\]
has full (row) rank.
For the simplicity of computations and without loss of generality, we assume that the first three exponent vectors are linearly independent.
After a change of variables, the exponent matrix takes the simple form
\[
\mB = \begin{pmatrix} 
1 & 0 & 0 & 0 & b_1 & 0 \\
0 & 1 & 0 & 0 & b_2 & 0 \\
0 & 0 & 0 & 1 & b_3 & 0
\end{pmatrix} .
\]
That is, we actually consider 
\begin{align*}
c_1 \, x_1 + c_2 \, x_2 - 1 &= 0 , \\ 
c_3 \, x_3 + c_4 \, x_1^{b_1} x_2^{b_2} x_3^{b_3} - 1 &= 0 . \\ 
\end{align*}
Crucial exponent data are
\[
\mBp = \begin{pmatrix} 
1 & 0 & 0 & 0 & b_1 & 0 \\
0 & 1 & 0 & 0 & b_2 & 0 \\
0 & 0 & 0 & 1 & b_3 & 0 \\
1 & 1 & 1 & 0 & 0 & 0 \\
0 & 0 & 0 & 1 & 1 & 1
\end{pmatrix} 
\quad \text{and} \quad
D = \ker \mBp = \im z
\]
with
\[
z=
\begin{pmatrix} b_1 \\ b_2 \\ -(b_1+b_2) \\ b_3 \\ -1 \\ 1 - b_3 \end{pmatrix} .
\]
In particular, $\dep = \dim D = 1$.
Now, we can consider the binomial equation on the coefficient polytope, 
\[
y^z = \cc^z ,
\quad \text{that is,} \quad
\lambda_1^{b_1} (1-\lambda_1)^{b_2} \lambda_2^{b_3} (1-\lambda_2)^{-1} = c_1^{b_1} c_2^{b_2} c_3^{b_3} c_4^{-1} =: c^* .
\]
After separating variables, we find
\begin{equation*} \label{eq:exa:intro} 
\lambda_1^{b_1} (1-\lambda_1)^{b_2} = c^* \lambda_2^{-b_3} (1-\lambda_2)^{1} ,
\end{equation*}
and, using sign-characteristic functions, we write
\[
s_{b_1,b_2}(\lambda_1)  = c^* s_{-b_3,1}(\lambda_2) .
\]
(For the definition of sign-characteristic functions $s$ and their roots (inverses) $r$, see Appendix~\ref{app:sc}.)

We distinguish three cases regarding the signs of the exponents.
\begin{enumerate}
\item
If $b_1 \cdot b_2 < 0$ or $b_3>0$, then at least one of the two sign-characteristic functions is strictly monotonic.
If $s_{-b_3,1}$ is monotonic and hence invertible,
we obtain the parametrization
\[
\lambda_2 = r_{-b_3,1} \left( s_{b_1,b_2}(\lambda_1) / c^* \right) ,
\]
having one component, for all $c^*$.

\begin{minipage}{0.93\textwidth}
We plot the parametrization for $b_3=2$ and $b_1=1,b_2=2$ (and $c^* = 1$).
\begin{center}
\includegraphics[width=0.3\textwidth]{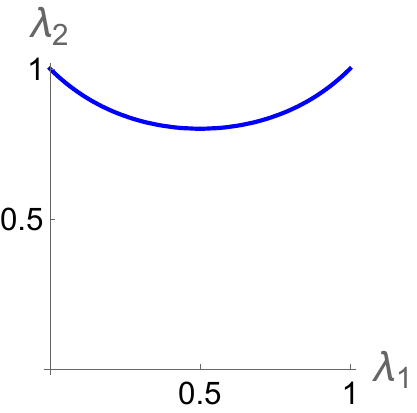}
\end{center}
\end{minipage}

If $s_{b_1,b_2}$ is monotonic and hence invertible,
we obtain the parametrization
\[
\lambda_1 = r_{b_1,b_2} \left( c^* s_{b_1,b_2}(\lambda_2) \right) .
\]
\item
If $b_1, b_2 > 0$ and $b_3<0$, then both sign-characteristic functions have a maximum.
For
\[
c^* \neq \frac{s_{b_1,b_2}^{\max}}{s_{-b_3,1}^{\max}} ,
\]
that is, for generic $c^*$,
the parametrization has two components.

\begin{minipage}{0.93\textwidth}
We plot the parametrization for $b_3=-2$ and $b_1=1,b_2=2$.
Since $s_{b_1,b_2}^{\max} = s_{-b_3,1}^{\max} = \frac{4}{27}$,
the two components intersect for $c^*=1$.
\begin{center}
\includegraphics[width=0.3\textwidth]{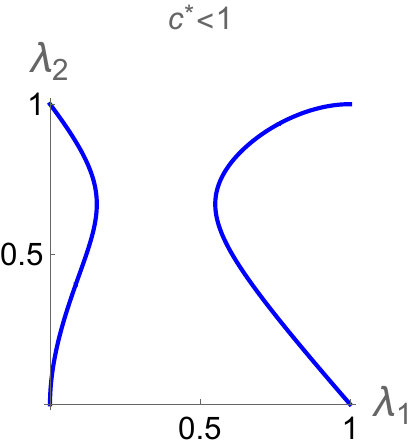} 
\includegraphics[width=0.3\textwidth]{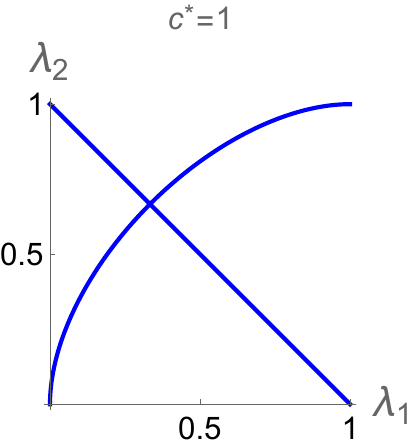}
\includegraphics[width=0.3\textwidth]{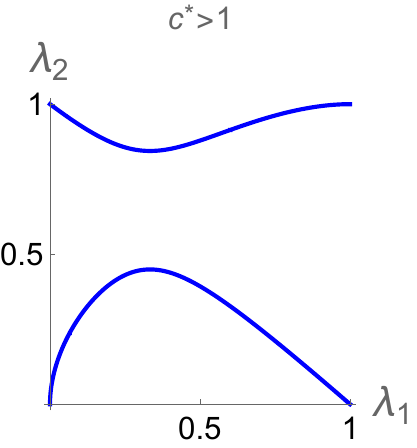}
\end{center}
\end{minipage}

\item
If $b_1, b_2 < 0$ and $b_3<0$, then the first sign-characteristic function has a minimum, and the second one has a maximum.
For
\[
c^* \ge \frac{s_{b_1,b_2}^{\min}}{s_{-b_3,1}^{\max}} ,
\]
that is, for large enough $c^*$, 
solutions exist, and the parametrization has one component.

\begin{minipage}{0.93\textwidth}
We plot the parametrization for $b_1=-1,b_2=-2$ and $b_3=-2$.
Since $s_{b_1,b_2}^{\min} = \frac{27}{4}$ and $s_{-b_3,1}^{\max} = \frac{4}{27}$,
solutions exist for $c^* \ge c^\text{crit} := (\frac{27}{4})^2$.
\begin{center}
\includegraphics[width=0.3\textwidth]{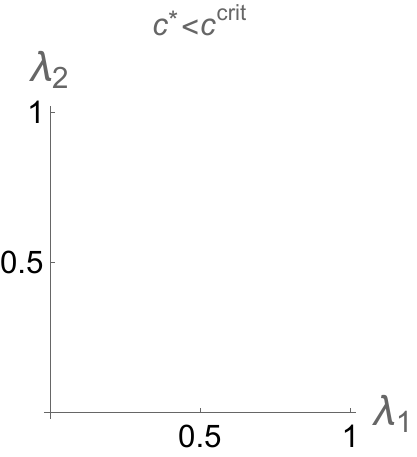} 
\includegraphics[width=0.3\textwidth]{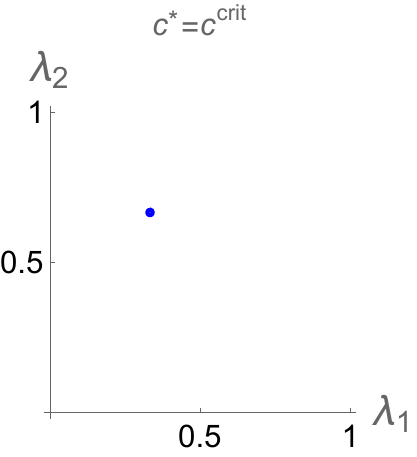}
\includegraphics[width=0.3\textwidth]{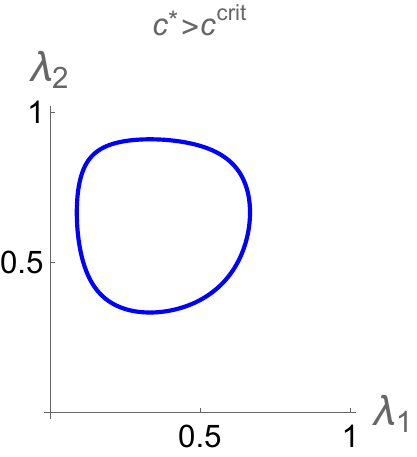}
\end{center}
\end{minipage}
\end{enumerate}

The parametrizations on the (open) unit square $(0,1)^2$ yield parametrizations of $Y_\cc$ (the solution set on the coefficient polytope)
and, ultimately, parametrizations of $Z_\cc$ (the set of positive solutions to $\mA \, (\cc \circ x^\mB)=0$).
\end{exa}


\begin{exa} \label{exa:Bihan}
We consider {\bf two overlapping trinomials involving four monomials in two variables},
that is,
\[
\mA \, x^\mB = 0
\quad\text{with}\quad
\mA ,\mB \in \R^{2 \times 4} .
\]
Clearly, $\mm=4$, $n=2$, and we assume $\ell=1$ and the generic case, that is, $\dim L = 2$.
Hence, $d = m - \ell- \dim L = 1$.

By our global assumption, $\mA$ has full row rank and $C = \ker \mA \cap \R^4_> \neq \emptyset$.
Hence, $\overline C = \ker \mA \cap \R^4_\ge$ is generated by two support-minimal, non-negative vectors.
Note that $\overline C$ is an {\em s-cone} (subspace cone), arising from a linear subspace and non-negativity conditions,
cf.~\cite{MuellerRegensburger2016}.
Further, the supports of the two generators overlap.
(Otherwise, there are $\ell=2$ classes. 
Then, $\dim \setP=0$ and $\dep=0$,
and there exists exactly one solution, by Corollary~\ref{cor:d0}.)

Explicitly, let $\bar y^1,\bar y^2 \in \R^4_\ge$ be two non-proportional, support-minimal vectors of~$\overline C$.
They are the vertices of $\overline \setP$, the closure of a coefficient polytope.
In terms of the sum and difference vectors
\[
\yc = \frac{\bar y^1 + \bar y^2}{2} > 0 
\quad\text{and}\quad
\yp = \frac{\bar y^1 - \bar y^2}{2} ,
\]
the polytope can be written as
\[
\setP = \{ \yc + t \yp \mid t \in (-1,1) \} .
\]

Now, let
\[
q = \yp \circ \yc^{-1} \in \R^4 ,
\]
which has the following properties:
(i) $|q_i|\le1$;
(ii) there are $i\neq j \in \{1,2,3,4\}$ such that $\bar y^1_i=\bar y^2_j=0$
and hence $q_i=1$ and $q_j=-1$;
and (iii) the order on the components of $q$ does not depend on positive scalings of the vertices $\bar y^1$ and $\bar y^2$.

Without loss of generality, we assume that the components of $q$ are in decreasing order,
that is,
\[
q = \begin{pmatrix} 1 \\ q_2 \\ q_3 \\ -1 \end{pmatrix} 
\]
with $1 = q_1 \ge q_2 \ge q_3 \ge q_4 = -1$.

%
This concludes the analysis of the coefficient matrix, and we turn to the exponent matrix.

Since $\dep = 1$,
\[
D = \ker \mBp = \ker \begin{pmatrix} \mB \\ \ones_4^\trans \end{pmatrix} = \im b
\quad \text{with nonzero}\quad
b \in \R^4 .
\]

Now, we can consider the binomial equation on the coefficient polytope, 
\[
y^b = 1.
\]
Equivalently, $(y \circ \yc^{-1})^b = \yc^{-b}$ with $y \circ \yc^{-1} = (\yc + t \yp) \circ \yc^{-1} = 1_4 + t q$, that is,
\[
f(t) = c^*
\]
with
\begin{align*}
f(t) & := (1_4 + t q)^b = \prod_{i=1}^4 (1 + t q_i)^{b_i} , \\
c^* & := \yc^{-b} = \prod_{i=1}^4(\yc_i)^{-b_i} .
\end{align*}
In order to bound the number of solutions to $f(t)=c^*$,
we analyze the function $f \colon (-1,1) \to \R_>$.
Most importantly, the function
\[
g(t) := \frac{f'(t)}{f(t)} = \sum_{i=1}^4 \frac{b_i q_i}{1 + t q_i}
\]
has the same zeros as $f'(t)$. 
By Rolle's theorem,
\[
Z(f) \le 1 + Z(g) ,
\]
where $Z(f)$ denotes the number of distinct real zeros of a function $f$ (on a given interval).

Now, recall that $1_4 \cdot b = 0$.
Using partial sums of $b$,
we rewrite the function $g$ in order to achieve two goals:
(1) to demonstrate that $g$ has the same zeros as a univariate polynomial of degree at most two;
and (2) to obtain an optimal Descartes' rule of signs (in terms of the partial sums),
cf.~\cite[Part~V, Chap.~1, \S~1, Prob.~21]{PolyaSzegoeD,PolyaSzegoe}.

Indeed, using the ``telescoping sum''
\begin{equation} \label{eq:telescope}
\sum_{i=1}^k \beta_i \gamma_i = \sum_{i=1}^{k-1} \left( \textstyle \sum_{j=1}^i \beta_j \right) \left( \gamma_i - \gamma_{i+1} \right)
\quad\text{if}\quad
\sum_{i=1}^k \beta_i = 0 ,
\end{equation}
cf.~\cite[Equation~6.2.12]{Feinberg1995a},
and introducing the family of functions
\begin{equation} \label{eq:fqq}
\begin{aligned}
f_{q,q'} \colon (-1,1) & \to \R_> , \\
t &\mapsto \frac{q}{1 + t q} - \frac{q'}{1 + t q'} = \frac{q-q'}{(1 + t q)(1 + t q')} , 
\end{aligned}
\end{equation}
for $q,q' \in [-1,1]$,
we find
\begin{equation} \label{eq:g}
g(t) = \sum_{i=1}^3 \left( \textstyle \sum_{j=1}^i b_j \right) f_{q_i,q_{i+1}}(t) .
\end{equation}
Explicitly,
\small
\begin{align*}
g(t) = & + b_1 \, \frac{q_1 - q_2}{(1 + t q_1)(1 + t q_2)} \\
& + (b_1+b_2) \, \frac{q_2 - q_3}{(1 + t q_2)(1 + t q_3)} \\
& + (b_1+b_2+b_3) \, \frac{q_3 - q_4}{(1 + t q_3)(1 + t q_4)} .
\end{align*}
\normalsize
(1) With $D(t):= \prod_{i=1}^4 (1 + t q_i)$, the function
\[
p_2(t) := g(t) \, D(t) 
\]
is a univariate polynomial of degree at most two with the same zeros as $g(t)$. 
Explicitly,
\small
\begin{align*}
p_2(t) = & + b_1 \, (q_1 - q_2)(1 + t q_3)(1 + t q_4) \\
& + (b_1+b_2) \, (q_2 - q_3)(1 + t q_1)(1 + t q_4) \\
& + (b_1+b_2+b_3) \, (q_3 - q_4)(1 + t q_1)(1 + t q_2) .
\end{align*}
\normalsize

(2) We apply
Descartes' rule of signs {\em for functions},
as stated in the famous 1925 book by P\'olya and Szeg\"o.
\begin{thm}[\cite{PolyaSzegoeD,PolyaSzegoe}, Part~V, Chap.~1, \S~7, Prob.~87, 90] \label{thm:Polya}
Let  $I\subseteq\R$ be an interval.
A sequence of analytic functions $f_1, \ldots, f_n \colon I \to \R$
obeys Descartes' rule of signs if and only if, for every subsequence of indices ${1\le i_1 \le \ldots \le i_k \le n}$,
the Wronskian
\[
W(f_{i_1},\ldots,f_{i_k}) \neq 0
\]
and, for every two subsequences of indices $1\le i_1 \le \ldots \le i_k \le n$ and $1\le j_1 \le \ldots \le j_k \le n$ of the same size,
\[
W(f_{i_1},\ldots,f_{i_k}) \cdot W(f_{j_1},\ldots,f_{j_k}) > 0 .
\]
\end{thm}

Indeed,
strict inequalities $1 = q_1 > q_2 > q_3 > q_4 = -1$ imply that the sequence of functions $f_{q_1,q_2},f_{q_2,q_3},f_{q_3,q_4}$,
defining the function~$g$ via~Equation~\eqref{eq:g},
obeys Descartes' rule of signs.
Here, we just consider the case of two functions.
If $q_1>q_2>q_3$, then
\[
W(f_{q_1,q_2},f_{q_2,q_3}) = \frac{(q_1-q_2)(q_1-q_3)(q_2-q_3)}{(1+tq_1)^2(1+tq_2)^2(1+tq_3)^2} 
= f_{q_1,q_2} \, f_{q_1,q_3} \, f_{q_2,q_3}
> 0 .
\]
For the general case, see Lemma~\ref{lem:Descartes} in Class~\ref{cla:Bihan}.
Under the assumption of finitely many solutions,
\[
N(g) \le \sgnvar \, (b_1, \, b_1+b_2, \, b_1+b_2+b_3) ,
\]
by Descartes' rule of signs.
Here, $N(g)$ denotes the number of real zeros counted with their multiplicities of a function $g$ (on a given interval).
If consecutive~$q_i$ are equal, then there are less than three terms in the sequence.
For example,
if $q_1=q_2 \neq q_3$, then $N(g) \le \sgnvar \, (b_1+b_2, \, b_1+b_2+b_3)$.
Note that the rule of signs depends on~$b$ (arising from the exponents),
but also on the order of~$q$ (arising from the coefficients).
Altogether, we have shown the following result.

\begin{pro} \label{pro:exa}
Let $\mA \, x^\mB = 0$ with $\mA,\mB \in \R^{2\times4}$
be two overlapping trinomials involving four monomials in two variables (after Gaussian elimination).

Assume that
(i) $\ker \mA$ is not a direct product, and
(ii) $\ker \binom{\mB}{\ones_4}$ has dimension one.
In particular, let $\bar y^1, \bar y^2 \in (\ker \mA \cap \R^4_\ge)$ be the vertices of a coefficient polytope, 
let $q = (\bar y^1 - \bar y^2) \circ (\bar y^1 + \bar y^2)^{-1} \in \R^4$,
and assume that $1 = q_1 \ge q_2 \ge q_3 \ge q_4 = -1$ (after reordering of columns).
Further, let $\ker \binom{\mB}{\ones_4} = \im b$ with $b \in \R^4$. 

Now, let $n_{\mA,\mB}$ denote the number of distinct positive solutions $x \in \R^2_>$ to ${\mA \, x^\mB = 0}$,
and assume that $n_{\mA,\mB}$ is finite.
Then, 
\[
n_{\mA,\mB} \le 1 + Z(p_2)
\]
and
\[
Z(p_2) \le \sgnvar \, (b_1, \, b_1+b_2, \, b_1+b_2+b_3) ,
\]
where ${p_2 \colon (-1,1) \to \R}$ is a univariate polynomial of degree at most two defined by $q$ and $b$.
\end{pro}
\begin{proof}
We just recall the main steps.
By Theorem~\ref{thm:main}, $n_{\mA,\mB} = Z(f)$.
By Rolle's theorem, $Z(f) \le 1 + Z(f')$. 
By the definitions of $g$ and $p_2$, $Z(f')=Z(g)=Z(p_2)$.
By the definitions of numbers of distinct zeros and numbers of zeros counted with their multiplicities, $Z(g) \le N(g)$.
By Descartes' rule of signs for functions, $N(g) \le \sgnvar \, (b_1, \, b_1+b_2, \, b_1+b_2+b_3)$.
\end{proof}

\paragraph{Concrete systems.}

\newcommand{\aaa}{s}
\newcommand{\bbb}{r}

For $\bbb > 1$ and $\aaa>0$, we consider the family of two overlapping trinomials involving four monomials in two variables,
\begin{alignat*}{3}
(\bbb/\aaa) \, x^2 &- \, x \, y^2 &&+1 &&= 0 , \\
(1/\aaa) \, x^2 && - \, y & +1 &&=0 .
\end{alignat*}
In compact form,
$\mA \, x^\mB = 0$ with
\[
\mA = \begin{pmatrix} \bbb/\aaa & -1 & 0 & 1 \\ 1/\aaa & 0 & -1 & 1 \end{pmatrix} 
\quad \text{and} \quad
\mB = \begin{pmatrix} 
2 & 1 & 0 & 0 \\ 0 & 2 & 1 & 0
\end{pmatrix} .
\]

\underline{Coefficients:}
Let $\bar y^1, \bar y^2 \in (\ker \mA \cap \R^4_\ge)$ be the vertices of a coefficient polytope~$\overline P$,
in particular, let
\begin{align*}
\bar y^1 &= (\aaa,\bbb,1,0)^\trans , \\
\bar y^2 &= (0,1,1,1)^\trans ,
\end{align*}
and hence 
\[
q = (\bar y^1- \bar y^2) \circ (\bar y^1 + \bar y^2)^{-1} = \left( 1, q_2, 0, -1 \right)^\trans 
\]
with 
\[ 0 < q_2 = \frac{\bbb-1}{\bbb+1} < 1 . 
\]

\underline{Exponents:}
The dependency subspace is given by
\[
\ker \binom{\mB}{\ones_4} = \im b
\quad \text{with} \quad b = (1,-2,4,-3)^\trans 
\]
and hence
\[
\sgnvar (b_1,b_1+b_2,b_1+b_2+b_3) = \sgnvar (1,-1,3) = 2 .
\]

After scaling, the binomial equation on the coefficient polytope is given by
\[
f(t) = c^*
\]
with
\begin{align*}
f(t) &:= (1_4 + t q)^b = (1+t)^1 (1+ t q_2)^{-2} \, 1^4 \, (1-t)^{-3} , \\
c^* &:= (\bar y^1+\bar y^2)^{-b} = \aaa^{-1} (\bbb+1)^2 \, 2^{-4} \, 1^3 .
\end{align*}

By Rolle's theorem and Descartes' rule of signs for functions, 
\[
Z(f) \le 1 + \sgnvar (b_1,b_1+b_2,b_1+b_2+b_3) = 1 + 2 = 3 .
\]
However, the univariate quadratic polynomial $p_2$ may yield a smaller bound,
\[
Z(f) \le 1 + Z(p_2) .
\]
Indeed, $p_2 = f' / f \cdot D$ with $D = \prod_{i=1}^4 (1 + q_i t)$ amounts to
\[
p_2(t) = 2 \left( 2q_2 \, t^2 + (2q_2+1) t + 2 -q_2 \right)
\]
with
\[
Z(p_2) = \begin{cases} 0 \\ 1 , \\ 2 \end{cases} 
\text{if } q_2 \begin{cases} < \\ = 1/2 + 1/\sqrt 6 . \\ > \end{cases} 
\]
We distinguish the following cases:
\begin{itemize}
\item
If $\bbb<11+4\sqrt 6 \approx 20.798$, then $q_2 < 1/2+1/\sqrt 6 \approx 0.908$, $Z(p_2)=0$, and $Z(f) \le 1$.

Considering the behavior of $f$ at the boundaries of $(-1,1)$,
we conclude that $Z(f)=1$.
\item
If $\bbb>11+4\sqrt 6$, then $q_2> 1/2+1/\sqrt 6$, $Z(p_2)=2$, and $Z(f) \le 3$.

For numerical computations, we set $\bbb=21 > 11+4\sqrt 6$. 
Then, $q_2=10/11$,
and the number of solutions depends on $\aaa$.
\begin{itemize}
\item[]
If $\aaa \le 65.609$ or $\aaa \ge 65.653$, then $Z(f)=1$.
\item[]
If $65.611 \le \aaa \le 65.652$, then $Z(f)=3$.
\end{itemize}
\end{itemize}

\end{exa}


\begin{cla} \label{cla:Bihan}

We consider the class of {\bf $n$ overlapping trinomials involving $(n+2)$ monomials in $n$ variables},
that is,
\[
\mA \, x^\mB = 0
\quad\text{with}\quad
\mA ,\mB \in \R^{n \times (n+2)} .
\]

In Theorem~\ref{thm:exa_general},
we bound the number of positive solutions
by one plus the number of roots of a univariate polynomial of degree at most $n$ on the interval $(-1,1)$.
This number is always less than or equal to the number of sign changes in an optimal Descartes' rule given in Bihan et al.~\cite[Theorem~2.4]{Bihan2021}.

We provide a statement of the result that is easily applicable (with all relevant objects defined within the theorem),
thereby also offering a simple proof for the Descartes' rule using our approach.
In particular, our parametrization of the coefficient polytope gives rise to a special family of univariate functions, 
having simple Wronskians and obeying Descartes' rule of signs.
Theorem~\ref{thm:exa_general} allows for exponents in~$\R$ 
with affine dependencies, %
whereas \cite[Theorem~2.4]{Bihan2021} is formulated for exponents in~$\Z$ and ``circuits''.
(In the introduction, Bihan et al.\ state that their results are valid also in the case of real exponents.
They also refer to systems that are not circuits as the easier case.)

Proposition~\ref{pro:exa} in Example~\ref{exa:Bihan} deals with the special case ${n=2}$,
and the proof of Theorem~\ref{thm:exa_general} follows the line of argument in Example~\ref{exa:Bihan}.
Indeed, we first state a fact we have already used in the case $n=2$,
namely,
that sequences of the functions $f_{q,q'}$ defined in Equation~\eqref{eq:fqq}
satisfy the conditions of Theorem~\ref{thm:Polya}.
\begin{lem} \label{lem:Descartes}
For $q,q' \in [-1,1]$, let $f_{q,q'} \colon (-1,1) \to \R_> ,\, t \mapsto \frac{q}{1 + t q}-\frac{q'}{1 + t q'}$.
Further, let $q_1, \ldots, q_{n+2} \in [-1,1]$ and $q_1 > \ldots > q_{n+2}$.
Then the sequence of functions $f_{q_1,q_2}, \ldots, f_{q_{n+1},q_{n+2}}$ obeys Descartes' rule of signs.
In particular,
\[
W(f_{q_1,q_2},\ldots,f_{q_k,q_{k+1}}) = 
\prod_{i=1}^{k} \, (i-1)! \prod_{j=i+1}^{k+1} f_{q_i,q_j}
> 0 .
\]
\end{lem}

\begin{proof}
By induction.
\end{proof}

Now,
we can present our new upper bound.
\begin{thm} \label{thm:exa_general}
Let $\mA \, x^\mB = 0$ with $\mA,\mB \in \R^{n\times(n+2)}$
be $n$ overlapping trinomials involving $n+2$ monomials in $n$ variables (after Gaussian elimination).

Assume that
(i) $\ker \mA$ is not a direct product, and
(ii) $\ker \binom{\mB}{\ones_{n+2}}$ has dimension one.
In particular, let $\bar y^1, \bar y^2 \in (\ker \mA \cap \R^{n+2}_\ge)$ be the vertices of a coefficient polytope, 
let $q = (\bar y^1 - \bar y^2) \circ (\bar y^1 + \bar y^2)^{-1} \in \R^{n+2}$,
and assume that $1 = q_1 \ge \ldots \ge q_{n+2} = -1$
(after reordering of columns).
Further, let $\ker \binom{\mB}{\ones_{n+2}} = \im b$ with $b \in \R^{n+2}$,
and let $s_i = \sum_{j=1}^i b_j$, $i = 1,\ldots, n+1$, be partial sums.

Now, let $n_{\mA,\mB}$ be the number of distinct positive solutions $x \in \R^n_>$ to ${\mA \, x^\mB = 0}$,
and assume that $n_{\mA,\mB}$ is finite.
Then, 
\[
n_{\mA,\mB} \le 1 + Z(p_n)
\]
with
\[
Z(p_n) \le \sgnvar \left( s_i \right)_{i=1}^{n+1} ,
\]
where ${p_n \colon (-1,1) \to \R}$ is a univariate polynomial of degree at most~$n$, defined by
\[
p_n(t) = \sum_{i=1}^{n+1} s_i \, (q_{i+1}-q_i) \prod_{\substack{j \in \{1,\ldots,n+2\} \\ \setminus \{i,i+1\}}} (1+t q_j) .
\]
If (consecutive) components of $q$ are equal, 
the latter bound can be further refined.
Let $I_1,\ldots,I_k \subset \{1,\ldots,n+2\}$ be the corresponding equivalence classes of indices (with an obvious order),
let $\tilde b \in \R^k$ with $\tilde b_i = \sum_{j \in I_i} b_j$,
and let $\tilde s_i = \sum_{j=1}^i \tilde b_j$, $i = 1,\ldots, k-1$, be partial sums.
Then, 
\[
Z(p_n) \le  \sgnvar \left( \tilde s_i \right)_{i=1}^{k-1} .
\]
\end{thm}
\begin{proof}
We follow the line of argument in Example~\ref{exa:Bihan}.

For orientation, we start with a classification of the system.
Clearly, $\mm=n+2$; by assumption (i), $\ell=1$; and by assumption (ii), $\dep=1$.
Hence, $\dim \LL = \mm-\ell-\dep = n$, and the system is generic.

The coefficient polytope can be written as
$
\setP = \{ \yc + t \yp \mid t \in (-1,1) \}
$
with 
$
\yc = \frac{\bar y^1 + \bar y^2}{2} > 0
$
and
$\yp = \frac{\bar y^1 - \bar y^2}{2}
$,
the monomial dependency subspace is given by $D = \im b$,
and hence the binomial equation on the coefficient polytope amounts to $y^b = 1$.
Equivalently, $(y \circ \yc^{-1})^b = \yc^{-b}$ with $y \circ \yc^{-1} = (\yc + t \yp) \circ \yc^{-1} = 1_4 + t q$, that is,
\[
f(t) = c^*
\]
with
\begin{align*}
f(t) & := (1_{n+2} + t q)^b = \prod_{i=1}^{n+2} (1 + t q_i)^{b_i} , \\
c^* & := \yc^{-b} = \prod_{i=1}^{n+2}(\yc_i)^{-b_i} .
\end{align*}
In order to bound the number of solutions to $f(t)=c^*$,
we analyze the function $f \colon (-1,1) \to \R_>$.
Most importantly, the function
\[
g(t) := \frac{f'(t)}{f(t)} = \sum_{i=1}^{n+2} \frac{b_i q_i}{1 + t q_i}
\]
has the same zeros as $f'(t)$.
Using the ``telescoping sum''~\eqref{eq:telescope},
the family of functions introduced in Equation~\eqref{eq:fqq},
and the partial sums defined above,
we find
\begin{align*}
g(t) &= \sum_{i=1}^{n+2} b_i \, \frac{q_i}{1 + t q_i} \\
&= \sum_{i=1}^{n+1} \left( \textstyle \sum_{j=1}^i b_j \right) \left( \frac{q_i}{1 + t q_i} - \frac{q_{i+1}}{1 + t q_{i+1}} \right) \\
&= \sum_{i=1}^{n+1} \left( \textstyle \sum_{j=1}^i b_j \right) \left( q_{i+1} - q_i \right) \frac{1}{(1 + t q_i)(1 + t q_{i+1})} \\
&= \sum_{i=1}^{n+1} s_i \, f_{q_i,q_{i+1}}(t) .
\end{align*}
On the one hand, with $D(t):=(1_{n+2} + t q)^{1_{n+2}}=\prod_{i=1}^{n+2} (1+t q_i)$, the function
\[
p_n(t) := g(t) \, D(t) = \sum_{i=1}^{n+1} s_i \, (q_{i+1}-q_i) \prod_{\substack{j \in \{1,\ldots,n+2\} \\ \setminus \{i,i+1\}}} (1+t q_j)
\]
is a univariate polynomial of degree at most $n$
and has the same zeros as $g$.
By Rolle's theorem,
$
Z(f) \le 1 + Z(f')
$,
and $Z(f')=Z(g)=Z(p_n)$ implies
\[
Z(f) \le 1 + Z(p_n) .
\]

On the other hand,
by Lemma~\ref{lem:Descartes},
strict inequalities
$q_1 > \ldots > q_{n+2}$
imply that the sequence of functions $f_{q_1,q_2}, \ldots, f_{q_{n+1},q_{n+2}}$ obeys Descartes' rule of signs,
cf.~Theorem~\ref{thm:Polya}.
Hence, under the assumption of finitely many solutions,
the number of zeros of $g$ counted with their multiplicities is bounded by the number of sign changes in the sequence
$
(s_i)_{i=1}^{n+1}
$.
Now, $Z(g) = Z(p_n)$ implies
\[
Z(p_n) \le \sgnvar \left( s_i \right)_{i=1}^{n+1} .
\]

If (consecutive) components of $q$ are equal, then
\[
g(t) = \sum_{i=1}^{n+2} b_i \, \frac{q_i}{1 + t q_i} 
= \sum_{i=1}^{k} \sum_{j \in I_i} b_j \, \frac{q_j}{1 + t q_j}
= \sum_{i=1}^{k} \tilde b_i \, \frac{q_{r(i)}}{1 + t q_{r(i)}} ,
\]
where $r(i) \in I_i$ is a representative of the index set $I_i$.
We can apply the same argument as above
and obtain $Z(p_n) \le \sgnvar \left( \tilde s_i \right)_{i=1}^{k-1}$.
Clearly, $\sgnvar \left( \tilde s_i \right)_{i=1}^{k-1} \le \sgnvar \left( s_i \right)_{i=1}^{n+1}$.
\end{proof}


{\em Remark.}
The number $Z(p_n)$ of distinct roots of $p_n$ in the interval $(-1,1)$ can be computed by Sturm's theorem or real root isolation algorithms,  see e.g.~\cite{Basu2006}.
\end{cla}


\subsection{Dependency two and higher}

\begin{cla} \label{cla:tri-t}
We consider {\bf one trinomial and one $t$-nomial {\rm (with $t\ge3$)} in two variables}, that is,
\begin{align*}
\bar c_1 \, x^{\bar b^1} + \bar c_2 \, x^{\bar b^2} - 1 &= 0 , \\ 
\pm c_1 \, x^{b^1} \pm \ldots \pm c_{t-1} \, x^{b^{t-1}} - 1 &= 0 ,
\end{align*}
with $x \in \R^2_>$,
nonzero $\bar b^1, \bar b^2, b^1, \ldots, b^{t-1} \in \R^2$,
$\bar b^1 \neq \bar b^2$, and $b^i \neq b^j$ for $i \neq j$
as well as
$\bar c_1,\bar c_2, c_1,\ldots, c_{t-1}>0$.
In compact form,
$\mA \, (\cc \circ x^\mB)$ with
\begin{align*}
\mA &= \begin{pmatrix} 1 & 1 & -1 & 0 & \ldots & 0 & 0 \\ 0 & 0 & 0 & \pm 1 & \ldots & \pm 1 & -1 \end{pmatrix} , \\[2ex]
\mB &= \begin{pmatrix} \bar b^1 & \bar b^2 & 0 & \phantom{x} b^1 & \ldots & b^{t-1} & 0 \end{pmatrix} , \\[2ex]
\cc &= \begin{pmatrix} \bar c_1 & \bar c_2 & 1 & \phantom{x} c_1 & \ldots & c_{t-1} & 1 \end{pmatrix}^\trans .
\end{align*}
Clearly, $\mm=3+t$, $n=2$, and $\ell=2$. 
We assume the generic case, where
\[
M = \mB \, \begin{pmatrix} \IL_3 & 0 \\ 0 & \IL_t \end{pmatrix} = \begin{pmatrix} 
\bar b^1& \bar b^2 & b^1 & \ldots & b^{t-1}
\end{pmatrix}
\]
has full (row) rank.
Hence, $\dim \LL = n$
and $\dep = \mm - \ell - n = t-1 \ge 2$.

If $\bar b^1, \bar b^2$ are linearly dependent, then, after solving the first equation, 
the second equation becomes a $t$-nomial in one variable, having at most $t-1$ positive solutions.
Hence, we assume that $\bar b^1, \bar b^2$ are linearly independent,
and, after a change of variables, the exponent matrix takes the form
\[
\mB = \begin{pmatrix} 
1 & 0 & 0 & \alpha_1 & \ldots & \alpha_{t-1} & 0 \\
0 & 1 & 0 & \beta_1 & \ldots & \beta_{t-1} & 0 \\
\end{pmatrix} 
\]
with $(\alpha_i, \beta_i) \neq (0,0)$ and $(\alpha_i, \beta_i) \neq (\alpha_j, \beta_j)$ for $i \neq j$.
That is, we actually consider 
\begin{align*} 
\bar c_1 \, x_1 + \bar c_2 \, x_2 - 1 &= 0 , \\ 
\pm c_1 \, x_1^{\alpha_1} x_2^{\beta_1} \pm \ldots \pm c_{t-1} \, x_1^{\alpha_{t-1}} x_2^{\beta_{t-1}} - 1 &= 0 . 
\end{align*}

We proceed in two steps.
(i) We obtain a global upper bound in terms of $t$.
To this end,
we solve the first equation, and the second equation becomes a sum of $t$~sign-characteristic functions in one variable.
(ii) For the simplest case $t=3$ (two trinomials),
we refine the upper bound in terms of the exponents (of the second equation),
see Example~\ref{exa:tritri} below.

Considered separately, the first equation has monomial dependency $\dep=3-1-2=0$.
Our approach yields the explicit parametrization of positive solutions
\[
x_1 = \frac{\lambda}{\bar c_1} \quad \text{and} \quad x_2 = \frac{1-\lambda}{\bar c_2} \quad \text{with} \quad \lambda \in (0,1) .
\]
Inserting the solutions of the first equation into the second equation yields $f(\lambda)=0$ for a sum of sign-characteristic functions
\begin{equation} \label{eq:sum_sc}
\begin{aligned} 
f(\lambda) &=
\pm \gamma_1 \, \lambda^{\alpha_1} (1-\lambda)^{\beta_1} \pm \ldots \pm \gamma_{t-1} \, \lambda^{\alpha_{t-1}} (1-\lambda)^{\beta_{t-1}} - 1 \\
&\phantom{}= \pm \gamma_1 \, s_{\alpha_1,\beta_1}(\lambda) \pm \ldots \pm \gamma_{t-1} \, s_{\alpha_{t-1},\beta_{t-1}}(\lambda) - 1 ,
\end{aligned}
\end{equation}
where $\gamma_i = \frac{c_i}{\bar c_1^{\alpha_i} \bar c_2^{\beta_i}} > 0$.

Theorem~\ref{thm:tnomial} below improves previous upper bounds
for the number of (finitely many) positive solutions,
cf.~\cite[Theorem~1(a)]{Li2003} and \cite[Corollary~16]{Koiran2015a}.
It is based on the following result by Koiran et al.~(2015).

\begin{thm}[\cite{Koiran2015a}, Theorem~9] \label{thm:Koiran}
Let $f_1, \ldots, f_n$ be analytic, linearly independent functions on an interval $I$,
and let $W_i = W(f_1, \ldots, f_i)$ denote the Wronskian of the first $i$ functions.
Then,
\[
Z(f_1+...+f_n) \le n-1+Z(W_n)+Z(W_{n-1})+2 \sum^{n-2}_{i=1} Z(W_i) ,
\]
where $Z(g)$ denotes the number of distinct real zeros of a function $g$ on $I$.
\end{thm} 

As a main result, we obtain an upper bound for the general class.
\begin{thm} \label{thm:tnomial}
For a trinomial and a $t$-nomial (with $t\ge3$) in two variables, the number of positive solutions is infinite or bounded by
\[
\frac{1}{3} t^3 - t^2 + \frac{8}{3} t - 2 .
\]
\end{thm}
\begin{proof}
We consider the zeros of the function $f(\lambda)$ in Equation~\eqref{eq:sum_sc},
in particular, the sequence of sign-characteristic functions 
\[
s_{\alpha_1,\beta_1}, \; \ldots, \; s_{\alpha_{t-1},\beta_{t-1}} , \; 1 ,
\]
and we introduce corresponding Wronskians,
\[
W_i = W(s_{\alpha_1,\beta_1}, \ldots, s_{\alpha_i,\beta_i})
\]
for $1 \le i \le t$ (where $s_{\alpha_t,\beta_t} = s_{0,0} =1$).
%
%
By Theorem~\ref{thm:Koiran},
\[
Z(f) \le t-1+Z(W_t)+Z(W_{t-1})+2 \sum^{t-2}_{i=1} Z(W_t) ,
\]
by Proposition~\ref{pro:Wronskian} in Appendix~\ref{app:sc},
\[
Z(W_i) \le \binom{i}{2} ,
\]
and hence
\[
Z(f) \le t-1+ \binom{t}{2} + \binom{t-1}{2} + 2\sum^{t-2}_{i=1}\binom{i}{2} .
\]
By induction,
\[
\sum_{i=1}^k \binom{i}{2} = \underbrace{0+1+3+6+\ldots}_{k \, \text{terms}} =  \frac{(k-1)k(k+1)}{6} 
\]
and hence
\begin{align*}
Z(f) &\le
t-1+\frac{t(t-1)}{2}+\frac{(t-1)(t-2)}{2}+2\,\frac{(t-3)(t-2)(t-1)}{6} \\
&= t(t-1) + \frac{(t-3)(t-2)(t-1)}{3} \\
&= \frac{1}{3} t^3 - t^2 + \frac{8}{3} t - 2 .
\end{align*} 
\end{proof}

We compare the upper bounds on the number of (finitely many) positive solutions
to one trinomial and one $t$-nomial in two variables, given in Theorem~\ref{thm:tnomial} (of this work), \cite[Theorem~1(a)]{Li2003}, and \cite[Corollary~16]{Koiran2015a}.
\renewcommand{\arraystretch}{1.2}
\[
\begin{array}{|c|c|c|c|} \hline
& \text{this work} & \cite{Li2003} & \cite{Koiran2015a} \\ \hline
t & \frac{1}{3} t^3 - t^2 + \frac{8}{3} t - 2 & 2^t-2 & \frac{2}{3}t^3+5t \\ \hline \hline
3 & 6 & 6 & 33 \\
4 & 14 & 14 & 62\frac{2}{3} \\
5 & 28 & 30 & 108\frac{1}{3} \\
6 & 50 & 62 & 174 \\
\vdots & & & \\
10 & 258 & 1022 &716\frac{2}{3} \\
\vdots & & & \\ \hline
\end{array}
\]
\end{cla}

Very recently, our upper bounds were further improved by El-Hilany and Tavenas~\cite{ElHilanyTavenas2024}, 
using our factorization of the Wronskians $W_i$, see Proposition~\ref{pro:Wronskian} in Appendix~\ref{app:sc},
along with a nontrivial result based on {\em dessins d’enfant} to bound the number $Z(W_t)+Z(W_{t-1})$,
see Theorem~\ref{thm:Koiran}. 

{\em Remark.}
For the number of real intersection points of a line with a plane curve defined by a (standard) polynomial with at most $t$ monomials, 
there is linear upper bound. More specifically, improving a result by Avenda\~{n}o~\cite{Avendano2009}, Bihan and El Hilany~\cite{BihanElHilany2017} prove that there are at most $6t-7$ real solutions and show that the upper bound is sharp for $t=3$. However, the proof works only for polynomials and not for generalized polynomials (with real exponents). 


\begin{exa} \label{exa:tritri}
We consider {\bf two trinomials in two variables},
that is, the simplest case $t=3$ in Class~\ref{cla:tri-t}, 
and we refine the upper bound.
This problem includes ``Haas systems'' which 
served as counterexamples to Kouchnirenko’s Conjecture~\cite{Khovanskij1980,Sturmfels1998,Haas2002,Dickenstein2007}.

By the standardization process used in Class~\ref{cla:tri-t}, we obtain
\begin{align*} 
\bar c_1 \, x_1 + \bar c_2 \, x_2 - 1 &= 0 , \\ 
c_1 \, x_1^{\alpha_1} x_2^{\beta_1} + c_2 \, x_1^{\alpha_2} x_2^{\beta_2} - 1 &= 0 \nonumber
\end{align*}
and, after solving the first equation, we consider
\begin{equation}  \label{eq:sum_tritri}
f(\lambda) = \gamma_1 \, s_{\alpha_1,\beta_1}(\lambda) + \gamma_2 \, s_{\alpha_2,\beta_2}(\lambda) - 1 = 0 
\end{equation}
for $\lambda \in (0,1)$.
Recall $(\alpha_i, \beta_i) \neq (0,0)$ and $(\alpha_1, \beta_1) \neq (\alpha_2, \beta_2)$. 

The following result was shown in Li et al.~(2003).
We provide a much simpler proof.

\begin{thm}[cf.~\cite{Li2003}, Theorem 1] \label{thm:tritri}
For two trinomials in two variables, the number of positive solutions is infinite or bounded by five.
\end{thm}
\begin{proof}
We consider the zeros of the function $f$ defined in Equation~\eqref{eq:sum_tritri},
in particular, the sequence of sign-characteristic functions $1, s_{\alpha_1,\beta_1}, s_{\alpha_2,\beta_2}$,
and we introduce the corresponding Wronskians,
$W_1=W(1)=1$, $W_2=W(1, s_{\alpha_1,\beta_1})$, and $W_3=W(1, s_{\alpha_1,\beta_1}, s_{\alpha_2,\beta_2})$.
By Proposition~\ref{pro:Wronskian} in Appendix~\ref{app:sc}, $Z(W_3) \le 3$.

Essentially, we need to analyze three cases.
\begin{enumerate}
\item
$\alpha_1 \cdot \beta_1\le0$:

Then, 
$s_{\alpha_1,\beta_1}$ is strictly monotone,
$W_2 = W(1,s_{\alpha_1,\beta_1}) = s'_{\alpha_1,\beta_1} \neq 0$, and hence $Z(W_2) = 0$.
By Theorem~\ref{thm:Koiran}, 
\begin{align*}
Z(f) &\le 3-1+Z(W_3)+Z(W_2)+2 \, Z(W_1) \\
&= 2 + Z(W_3) \\
&\le 5 .
\end{align*}
The same argument applies for $\alpha_2 \cdot \beta_2\le0$.
\item
$\alpha_1 , \beta_1 > 0$ and $\alpha_2 , \beta_2 > 0$:

We consider 
\[
f'(\lambda) = \gamma_1 \, s'_{\alpha_1,\beta_1}(\lambda) + \gamma_2 \, s'_{\alpha_2,\beta_2}(\lambda)
\]
with
\[
s'_{\alpha_i,\beta_i}(\lambda) = s_{\alpha_i-1,\beta_i-1}(\lambda) \left( \alpha_i(1-\lambda)-\beta_i \lambda \right) 
\]
and
\[
\lambda^*_i = \frac{\alpha_i}{\alpha_i+\beta_i}, \quad 0 < \lambda^*_i < 1 .
\]
Hence, 
\[
s'_{\alpha_i,\beta_i}(\lambda)
\begin{cases}
\ge 0 & \text{for } \lambda \in (0,\lambda^*_i] , \\ 
\le 0 & \text{for } \lambda \in [\lambda^*_i,1) ,
\end{cases}
\]
and $s'_{\alpha_1,\beta_1}$ and $s'_{\alpha_2,\beta_2}$ have opposite signs
on the interval $(l,r)$ with $l = \min(\lambda^*_1,\lambda^*_2)$ and $r = \max(\lambda^*_1,\lambda^*_2)$
and only there.
(If $\lambda^*_1=\lambda^*_2$, that is, $\alpha_1/\beta_1=\alpha_2/\beta_2$,
then $f'$ has exactly one zero.)

By the properties of Wronskians, $W(s'_{\alpha_1,\beta_1},s'_{\alpha_2,\beta_2}) = W(1,s_{\alpha_1,\beta_1},s_{\alpha_2,\beta_2}) = W_3$.
By Theorem~\ref{thm:Koiran} with the interval $(l,r)$,
\begin{align*}
Z(f') &\le 2-1+Z(W(s'_{\alpha_1,\beta_1},s'_{\alpha_2,\beta_2}))+Z(W(s'_{\alpha_1,\beta_1})) \\
&= 1 + Z(W_3) \\
&\le 4 ,
\end{align*}
and by Rolle's Theorem, $Z(f) \le 5$.

The same argument applies for $\alpha_1 , \beta_1 < 0$ and $\alpha_2 , \beta_2 < 0$.
\item
$\alpha_1 , \beta_1 > 0$ and $\alpha_2 , \beta_2 < 0$:

Again, we consider
\[
f'(\lambda) = \gamma_1 \, s'_{\alpha_1,\beta_1}(\lambda) + \gamma_2 \, s'_{\alpha_2,\beta_2}(\lambda)
\]
which has the same zeros as
\[
g(\lambda) = \gamma_1 \, s_{\alpha_1-\alpha_2,\beta_1-\beta_2}(\lambda) \left( \alpha_1(1-\lambda)-\beta_1 \lambda \right) 
+ \gamma_2 \left( \alpha_2(1-\lambda)-\beta_2 \lambda \right) .
\]
Now,
\[
g'(\lambda) =  \gamma_1 \, s_{\alpha_1-\alpha_2-1,\beta_1-\beta_2-1}(\lambda) \cdot q_2(\lambda) - \gamma_2 \, (\alpha_2+\beta_2) 
\]
with the polynomial (of degree at most two)
\[
q_2(\lambda) = \left[ (\alpha_1-\alpha_2)(1-\lambda)-(\beta_1-\beta_2) \lambda \right]
(\alpha_1(1-\lambda) -\beta_1 \lambda) - \lambda (1-\lambda)(\alpha_1+\beta_1) .
\]
By Rolle's Theorem twice, $Z(f) \le Z(g')+2$.
Finally,
\begin{align*}
g''(\lambda) &= \gamma_1 \, s_{\alpha_1-\alpha_2-2,\beta_1-\beta_2-2}(\lambda) \cdot q_3 (\lambda) 
\end{align*}
with the polynomial (of degree at most three)
\begin{align*}
q_3(\lambda) 
=& + \alpha_1 (\alpha_1 - \alpha_2) (\alpha_1 - \alpha_2 - 1) \, (1-\lambda)^3 \\
& - (\alpha_1 - \alpha_2) [ 2 \alpha_1 (\beta_1 - \beta_2 + 1) + \beta_1 (\alpha_1 - \alpha_2 + 1) ] \, \lambda(1-\lambda)^2 \\
&+ (\beta_1 - \beta_2) [ \alpha_1 (\beta_1 - \beta_2 + 1) + 2 \beta_1 (\alpha_1 - \alpha_2 + 1) ] \, \lambda^2(1-\lambda) \\
& - \beta_1 (\beta_1 - \beta_2) (\beta_1 - \beta_2 -1) \, \lambda^3 .
\end{align*}
By Rolle's Theorem, $Z(g') \le Z(g'') + 1$.
Hence, $Z(f) \le Z(g') + 2 \le Z(g'') + 3$,
and the problematic case is $Z(g'')=3$.

In fact, we consider $\tilde q_3 \colon \R_> \to \R$ such that $q_3(\lambda)/(1-\lambda)^3 = \tilde q_3(\lambda/(1-\lambda))$.
That is, zeros of $q_3$ on $(0,1)$ correspond to zeros of $\tilde q_3$ on $(0,\infty)$. 
Now, $Z(g'')=3$ implies three sign changes in $\tilde q_3$ 
and hence $\alpha_1 - \alpha_2 - 1>0$ and $\beta_1 - \beta_2 -1>0$.
In this case,
\begin{gather*}
g'(0) = g'(1) = - \gamma_2 \, (\alpha_2+\beta_2) >0 
\end{gather*}
and
\begin{align*}
g''(0+) &= \gamma_1 \, s_{\alpha_1-\alpha_2-2,\beta_1-\beta_2-2}(0+) \cdot \alpha_1 (\alpha_1 - \alpha_2) (\alpha_1 - \alpha_2 - 1) > 0 .
\end{align*}
By Corollary~\ref{cor:Rolle++} in Appendix~\ref{app:Rolle} (for $g'$), $Z(g') \le Z(g'')-1 = 3-1=2$.
Ultimately, $Z(f) \le 4$.

The same argument applies for $\alpha_1 , \beta_1 < 0$ and $\alpha_2 , \beta_2 > 0$.
\end{enumerate}
\end{proof}

Finally, we refine the upper bound in terms of the exponents.
In fact, further case distinctions are possible.
Here we present a compromise between completeness and simplicity.
\begin{thm} \label{thm:tritri_exp}
Let $f \colon (0,1) \to \R$, 
\[
f(\lambda) = \gamma_1 \, s_{\alpha_1,\beta_1}(\lambda) + \gamma_2 \, s_{\alpha_2,\beta_2}(\lambda) - 1
\]
with 
$(\alpha_i, \beta_i) \neq (0,0)$, 
$(\alpha_1, \beta_1) \neq (\alpha_2, \beta_2)$, 
and $\gamma_1,\gamma_2>0$.
Then, $p_3 \colon (0,1) \to \R$,
\[
p_3(\lambda) = \frac{ W(1,s_{\alpha_1,\beta_1}(\lambda), s_{\alpha_2,\beta_2}(\lambda)) }{ s_{\alpha_1+\alpha_2-3,\beta_1+\beta_2-3}(\lambda) } ,
\]
is a polynomial of degree at most three,
and the following statements hold.
\begin{enumerate}[1.]
\item
If $\alpha_1 \cdot \beta_1<0$ and $\alpha_2 \cdot \beta_2>0$ (or vice versa), then $Z(f) \le 2 + Z(p_3) \le 5$. \\
If $\alpha_1 \cdot \beta_1<0$ and $\alpha_2 \cdot \beta_2<0$ or 
one of the exponents $\alpha_1,\beta_1, \alpha_2,\beta_2$ is zero, \\
then $Z(f) \le 2 + Z(p_3) \le 4$.
\item
If $\alpha_1,\beta_1>0$ and $\alpha_2,\beta_2>0$ (or both $<$), \\ 
then $Z(f) \le 4$ if $Z(p_3)=2$ or $3$, and $Z(f) \le 2$ if $Z(p_3)=0$ or $1$.
\item
If $\alpha_1,\beta_1>0$ and $\alpha_2,\beta_2<0$ (or vice versa), then $Z(f) \le 4$.
\end{enumerate}
As a consequence, 
if $Z(f) = 5$, then $\alpha_1 \cdot \beta_1<0$ and $\alpha_2 \cdot \beta_2>0$ (or vice versa) and $Z(p_3)=3$.
\end{thm}
\begin{proof}
By Proposition~\ref{pro:Wronskian} in Appendix~\ref{app:sc},
\[
W_3(\lambda) = s_{\alpha_1+\alpha_2-3,\beta_1+\beta_2-3}(\lambda) \cdot p_3(\lambda)
\]
with a polynomial $p_3 \colon (0,1) \to \R$ of degree at most three. Explicitly,
\begin{align*}
p_3(\lambda) 
=& - \alpha_1 \alpha_2 (\alpha_1 - \alpha_2) \, (1-\lambda)^3 
+ \ldots (1-\lambda)^2 \lambda + \ldots (1-\lambda) \lambda^2 \\
& + \beta_1 \beta_2 (\beta_1 - \beta_2) \, \lambda^3 ,
\end{align*}
where we only show the crucial coefficients.
Clearly, $Z(W_3) = Z(p_3) \le 3$.
%
Now, we consider the three statements.
\begin{enumerate}
\item
Let $\alpha_1 \cdot \beta_1\le0$ (or $\alpha_2 \cdot \beta_2\le0$).
By the corresponding case~1 in the proof of Theorem~\ref{thm:tritri},
$Z(f) \le 2 + Z(W_3) \le 5$.

For the refined bound, we consider $\tilde p_3 \colon \R_> \to \R$ such that $p_3(\lambda)/(1-\lambda)^3 = \tilde p_3(\lambda/(1-\lambda))$.
That is, zeros of $p_3$ on $(0,1)$ correspond to zeros of $\tilde p_3$ on $(0,\infty)$. 

If (i) one of the exponents $\alpha_1,\beta_1, \alpha_2,\beta_2$ is zero
or (ii) $\alpha_1 \cdot \beta_1<0$, $\alpha_2 \cdot \beta_2<0$, and additionally $\alpha_1 \cdot \alpha_2<0$,
then there are at most two sign changes in $\tilde p_3$ and hence $Z(W_3)\le2$.
If $\alpha_1 \cdot \alpha_2>0$ in (ii), then $s_{\alpha_1,\beta_1}, s_{\alpha_2,\beta_2}$ are strictly monotonic,
$s'_{\alpha_1,\beta_1} \cdot s'_{\alpha_2,\beta_2} > 0$, and $Z(f)=1$.
\item
Let $\alpha_1,\beta_1>0$ and $\alpha_2,\beta_2>0$ (or both $<$).
By the corresponding case~2 in the proof of Theorem~\ref{thm:tritri},
$Z(f') \le 1 + Z(W_3) \le 4$.
Further, $f(0)=f(1)=-1$ (or $f(0+)=f(1-)=\infty$).

By Corollary~\ref{cor:Rolle+} in Appendix~\ref{app:Rolle}, if $Z(f')$ is even, then $Z(f) \le Z(f') \le 4$.
By Rolle's Theorem, if $Z(f')$ is odd, then $Z(f) \le Z(f') + 1 \le 4$.

If $Z(W_3) \le 1$, then $Z(f') \le 2$.
By the same arguments as above, $Z(f) \le 2$.
\item
Let $\alpha_1 , \beta_1 > 0$ and $\alpha_2 , \beta_2 < 0$ (or vice versa).
By the corresponding case~3 in the proof of Theorem~\ref{thm:tritri},
$Z(f) \le 4$ if $Z(g'')=3$.
It remains to consider the cases $Z(g'')=2,1,0$.
Recall $Z(g) = Z(f')$, 
$Z(f) \le Z(g')+2$ by Rolle's Theorem twice,
and
\begin{align*}
g'(0+) &= \gamma_1 \, s_{\alpha_1-\alpha_2-1,\beta_1-\beta_2-1}(0+) \cdot (\alpha_1-\alpha_2) \, \alpha_1 
- \gamma_2 \, (\alpha_2+\beta_2) > 0 , \\
g'(1-) &= \gamma_1 \, s_{\alpha_1-\alpha_2-1,\beta_1-\beta_2-1}(1-) \cdot (\beta_1-\beta_2) \, \beta_1 
- \gamma_2 \, (\alpha_2+\beta_2) > 0 .
\end{align*}

By Corollary~\ref{cor:Rolle+} (for $g'$), if $Z(g'')=2$, then $Z(g') \le Z(g'') = 2$. 
Ultimately, $Z(f) \le 4$.

By Rolle's Theorem, if $Z(g'')\le1$, then $Z(g') \le Z(g'') + 1 \le 2$.
Ultimately, $Z(f) \le 4$.
\end{enumerate}
\end{proof}

Using Theorem~\ref{thm:tritri_exp}, we perform a search over small integer exponents
and easily find a counterexample to Kouchnirenko’s Conjecture~\cite{Khovanskij1980,Sturmfels1998}
that is even simpler than the smallest ``Haas system''~\cite{Haas2002,Dickenstein2007}.
Indeed,
\begin{equation}
\begin{aligned}
x^5/y + a \, y - 1 &= 0 , \\
y^5/x + b \, x - 1 &= 0 
\end{aligned}
\end{equation}
with $a=b=1.392$ has five positive solutions $(x,y)$. 

{\em Remark.}
In a personal communication, 
we have been informed by Maurice Rojas
that this example was also discovered (but not published) by Korben Rusek in 2013.
\end{exa}


\subsection*{Acknowledgements}

This research was funded in whole, or in part, by the Austrian Science Fund (FWF), 
grant DOIs 10.55776/P33218 and 10.55776/PAT3748324 to SM and grant DOI 10.55776/P32301 to GR.

\subsection*{Conflict of interest and data availability}

On behalf of all authors, the corresponding author states 
that there is no conflict of interest
and that the manuscript has no associated data.



\bibliographystyle{abbrv} 
\bibliography{MR,linearstability,polynomials,crnt}


\clearpage

\appendix 
\section*{Appendix}


\section{Sign-characteristic functions} \label{app:sc}

As a key technique for the analysis of trinomials,
we introduce a family of {\em sign-characteristic} functions
on the open unit interval.
For $\alpha,\beta \in \R$, let
\begin{align*}
s_{\alpha,\beta} \colon & (0,1) \to \R_>, \\ 
& \lambda \mapsto \lambda^\alpha (1-\lambda)^\beta .
\end{align*}
For $\alpha,\beta\neq0$,
\[
s'_{\alpha,\beta}(\lambda) = s_{\alpha-1,\beta-1}(\lambda) \left( \alpha(1-\lambda)-\beta \lambda \right) .
\]
Hence,
$s_{\alpha,\beta}$ has an extremum at
\[
\lambda^*= \frac{\alpha}{\alpha+\beta} \in (0,1)
\]
if and only if $\alpha \cdot \beta>0$.
Then, 
\[
s_{\alpha,\beta}(\lambda^*) = \left( \frac{\alpha}{\alpha+\beta} \right)^\alpha \left( \frac{\beta}{\alpha+\beta} \right)^\beta .
\]
We call the functions sign-characteristic since the signs ($-,0$ or $+$) of $\alpha$ and $\beta$
characterize the values ($0,1$ or $\infty$) at $0+$ and $1-$, respectively.
In particular,
\begin{itemize}
\item
if $\alpha,\beta>0$, then $s_{\alpha,\beta}(0+)=s_{\alpha,\beta}(1-)=0$, and $s_{\alpha,\beta}$ has a maximum at~$\lambda^*$,
\item
if $\alpha,\beta<0$, then $s_{\alpha,\beta}(0+)=s_{\alpha,\beta}(1-)=\infty$, and $s_{\alpha,\beta}$ has a minimum at~$\lambda^*$,
\item
if $\alpha>0>\beta$, then $s_{\alpha,\beta}(0+)=0, \, s_{\alpha,\beta}(1-)=\infty$, and $s_{\alpha,\beta}$ is strictly monotonically increasing, and
\item
if $\alpha<0<\beta$, then $s_{\alpha,\beta}(0+)=\infty, \, s_{\alpha,\beta}(1-)=0$, and $s_{\alpha,\beta}$ is strictly monotonically decreasing.
\end{itemize}

\subsection*{Roots of sign-characteristic functions}

On the monotonic parts of the sign-characteristic functions,
we introduce inverses or ``roots''.
In particular,
\begin{itemize}
\item
if $\alpha \cdot \beta<0$, then we define
\begin{align*}
r_{\alpha,\beta} \colon & \R_> \to (0,1) , \\
& \lambda \mapsto (s_{\alpha,\beta})^{-1}(\lambda) .
\end{align*}
\item
If $\alpha \cdot \beta>0$, then we denote the restrictions of $s_{\alpha,\beta}$ to $(0,\lambda^*]$ and $[\lambda^*,1)$
by $s_{\alpha,\beta}^-$ and $s_{\alpha,\beta}^+$, respectively, and we define
\begin{alignat*}{4}
r_{\alpha,\beta}^- \colon & (0,s_{\alpha,\beta}(\lambda^*)] \to (0,\lambda^*] , \quad\text{and}\quad && r_{\alpha,\beta}^+ \colon && (0,s_{\alpha,\beta}(\lambda^*)] \to [\lambda^*,1) , \\
& \lambda \mapsto (s_{\alpha,\beta}^-)^{-1}(\lambda) , && && \lambda \mapsto (s_{\alpha,\beta}^+)^{-1}(\lambda) .
\end{alignat*}
\end{itemize}
The cases $\alpha=0$ or $\beta=0$ can be treated analogously.

\subsection*{Wronskians of sign-characteristic functions} 

Wronskians of sign-characteristic functions have polynomials (with integer exponents) as factors.
In the simplest case of two functions,
\[
W(s_{\alpha_1,\beta_1},s_{\alpha_2,\beta_2}) = 
\begin{vmatrix}
s_{\alpha_1,\beta_1} & s_{\alpha_2,\beta_2} \\
s'_{\alpha_1,\beta_1} & s'_{\alpha_2,\beta_2}
\end{vmatrix}
= s_{\alpha_1+\alpha_2-1,\beta_1+\beta_2-1} \cdot p_{1}
\]
with the polynomial 
\[
p_{1}(\lambda) = (\alpha_2-\alpha_1)(1-\lambda)-(\beta_2-\beta_1) \lambda
\]
of degree at most one.


\newcommand{\nn}{n}

By induction,
the statement holds for arbitrary numbers of sign-characteristic functions.
\begin{pro} \label{pro:Wronskian}
For $\nn \in \N$,
the Wronskian of $\nn$ sign-characteristic functions $s_{\alpha_1,\beta_1}$, \ldots, $s_{\alpha_\nn,\beta_\nn}$
is given by
\[
W(s_{\alpha_1,\beta_1}, \ldots, s_{\alpha_\nn,\beta_\nn}) 
= s_{\bar \alpha-d, \bar \beta-d} \cdot p_d
\]
with 
\begin{align*}
\bar \alpha &= \alpha_1+\alpha_2+\ldots+\alpha_\nn , \\
\bar \beta &= \beta_1+\beta_2+\ldots+\beta_\nn ,
\end{align*}
\[
d = \binom{\nn}{2} ,
\]
and a polynomial $p_d$ of degree at most $d$.
\end{pro}
%



\section{Rolle's Theorem} \label{app:Rolle}

\begin{thm}[Rolle, anno domini 1691]
Let $f \colon [a,b] \to \R$ be continuous on the closed interval $[a,b]$ and differentiable on the open interval $(a,b)$.
If $f(a)=f(b)$, then there is $c \in (a,b)$ with $f'(c)=0$.
\end{thm}

Rolle’s theorem ensures that, for two zeros of $f$, there exists a zero of $f'$ that is strictly between them.
Let $Z(g)$ denote the number of zeros of $g \colon [a,b] \to \R$. 

\begin{cor} \label{cor:Rolle}
Let $f \colon [a,b] \to \R$ be continuous on the closed interval $[a,b]$ and differentiable on the open interval $(a,b)$.
Then, $Z(f) \le Z(f')+1$.
\end{cor}
\begin{proof}
The $Z(f)$ zeros of $f$ define $Z(f)-1$ open intervals with at least one zero of $f'$ in each of them.
That is, $Z(f') \ge Z(f)-1$.
\end{proof}

\clearpage

The result can be refined if the signs of the function (and its derivative) are known at the endpoints.

\begin{cor} \label{cor:Rolle+}
Let $f \colon [a,b] \to \R$ be continuous on the closed interval $[a,b]$ and differentiable on the open interval $(a,b)$.
\begin{enumerate}[(i)]
\item
If $Z(f')$ is even and $f(a) \cdot f(b)>0$, then $Z(f) \le Z(f')$.
\item
If $Z(f')$ is odd and $f(a) \cdot f(b)<0$, then $Z(f) \le Z(f')$.
\end{enumerate}
\end{cor}

\begin{cor} \label{cor:Rolle++}
Let $f \colon [a,b] \to \R$ be continuous on the closed interval $[a,b]$ and differentiable on the open interval $(a,b)$.
\begin{enumerate}[(i)]
\item
If $Z(f')$ is even, $f(a) \cdot f(b)<0$, and $f(a) \cdot f'(a)>0$ (or $f(a) \cdot f'(a+)=\infty$), then $Z(f) \le Z(f')-1$.
\item
If $Z(f')$ is odd, $f(a) \cdot f(b)>0$, and $f(a) \cdot f'(a)>0$ (or $f(a) \cdot f'(a+)=\infty$), then $Z(f) \le Z(f')-1$.
\end{enumerate}
\end{cor}

\end{document}